\documentclass[final]{siamltex1213}
\usepackage{stmaryrd}
\usepackage{latexsym,amssymb,mathrsfs,fancyhdr}
\usepackage{subfigure}
\usepackage{psfrag}
\usepackage{epsfig}
\usepackage{epstopdf}
\usepackage{pdfsync}
 \usepackage{hyperref}
\font\tenmsb=msbm10 \font\sevenmsb=msbm7 \font\fivemsb=msbm5
\catcode`\@=11 \ifx\amstexloaded@\relax\catcode`\@=\active
\endinput\else\let\amstexloaded@\relax\fi
\def\spaces@{\space\space\space\space\space}
\def\spaces@@{\spaces@\spaces@\spaces@\spaces@\spaces@}
\def\space@.{\futurelet\space@\relax} \space@.
\def\relaxnext@{\let\next\relax}
\def\accentfam@{7}
\def\tr{\triangle}
\def\a{\alpha}
\def\s{{\sigma}}
\def\Bn{{B_N}}
\def\t1{{\tau}}
\def\td{{\tau_1}}
\def\del{{\delta}}
\def\b{\beta}
\def\teta{{\vartheta}}
\def\tet1{{\theta}}
\def\noaccents@{\def\accentfam@{0}}
\def\mathcal{\relaxnext@\ifmmode\let\next\mathcal@\else
\def\next{\Err@{Use
\string\mathcal\space only in math mode}}\fi\next}
\def\mathcal@#1{{\mathcal@@{#1}}} \def\mathcal@@#1{\noaccents@\fam\tw@#1}
\def\frec#1{{\stackrel{#1}\rightarrow}}
\def\1{{\frac12}}

    \def\eqref #1{{(\ref{#1})}}
\def\Bbb{\relaxnext@\ifmmode\let\next\Bbb@\else
\def\next{\Err@{Use \string\Bbb\space only in math mode}}\fi\next}

\def\Bbb@#1{{\Bbb@@{#1}}}
\def\Bbb@@#1{\noaccents@\fam\msbfam#1}
\newfam\msbfam \textfont\msbfam=\tenmsb \scriptfont\msbfam=\sevenmsb
\scriptscriptfont\msbfam=\fivemsb
\def\N{{\mathbb N}} \def\Z{{\mathbb Z}} \def\ome{{\omega}} \def\Ome{{\Omega}} \def\ii{{\rm i}}
\def\R{{\mathbb R}} \def\T{{\mathbb T}}  \def\Ra{{\mathcal R}}
\def\C{{\mathbb C}} 
\def\D{{\mathcal D}}  \def\A{{\mathcal A}}

\newtheorem{Theorem}{Theorem}
\newtheorem{Lemma}{Lemma}[section]

\newtheorem{Proposition}{Proposition}

\newtheorem{Remark}{Remark}[section]
\newtheorem{remark}{Remark}[section]
\newtheorem{Example}{Example}[section]
\newtheorem{Definition}{Definition}[section]
\renewcommand{\theequation}{\thesection.\arabic{equation}}
\@addtoreset{equation}{section}
\newcommand{\sss}{\smallskip}

\newcommand{\bs}{\bigskip}

\newcommand{\la}{\langle } \newcommand{\ra}{\rangle }

\newcommand{\beq}{\begin{equation} } \newcommand{\eeq}{\end{equation} }

\newcommand{\n}{\mathfrak n}

\def\eps{{\varepsilon}}

\title{Quasi-T\"oplitz Functions in  KAM Theorem \thanks{Supported by the European Research Council under
FP7, project "Hamiltonian PDEs}}

\author{Michela Procesi\thanks{Universit\`a di Roma La Sapienza 00185 Roma Italy, mprocesi@mat.uniroma1.it}
        \and Xindong  Xu \thanks{Universit\`a degli Studi di Napoli "Federico II" 80126, Italy and   Department of Mathematics, Southeast University,  Nanjing 210096 China,  xindong.xu@seu.edu.cn\,,  xuxdnju@gmail.com}}
\begin{document}

% \date{put in custom date; delete this line for today's date}
% \maketitle

\maketitle
\begin{abstract} We define and describe the class of  Quasi-T\"oplitz functions. We then prove an abstract KAM theorem where the perturbation is in this class. We apply this theorem to a Non-Linear-Schr\"odinger equation on the torus $\T^d$, thus proving  existence and stability of quasi--periodic solutions.
\end{abstract}


\section{Introduction }

 In this paper,  we study  a model NLS with external parameters on the torus $\T^d$ and  prove existence and stability of quasi--periodic solutions. In order to do this we introduce a new class of functions, which we denote as quasi-T\"oplitz.  We focus on the equation
\begin{equation}\label{equ}
 iu_t-\tr u+{\mathbb M}_\xi
u+f(|u|^2)u=0,\qquad x\in \T^d,\ t\in \R,
\end{equation}
where $f(y)$ is a real analytic function with $f(0)=0$, while  ${\mathbb M}_\xi$ is a Fourier multiplier,
 namely a linear operator which commutes with the Laplacian and  whose role is to introduce $b$
  parameters in order to guarantee that equation \eqref{equ} linearized at $u=0$ admits a quasi--periodic solution with $b$ frequencies.
%\begin{eqnarray}\label{bound}
%&&u(t,x_1+2\pi,x_2\cdots,x_d)=\cdots=u(t,x_1,\cdots,x_{d-1},x_d+2\pi)
%\\&&=u(t,x_1,x_2,\cdots,x_d)\nonumber
%\end{eqnarray}
More precisely  we choose a
finite set   $\{\n^{(1)}=0,\n^{(2)},\cdots, \n^{(b)}\}$ with
$\n^{(i)}\in \Z^d$ and define ${\mathbb M}_\xi$ so that the eigenvalues of the
operator $\tr+{\mathbb M}_\xi$  are
\begin{equation}
\left\{\begin{array}{cclc}
\omega_j&=&|\n^{(j)}|^2+\xi_{j},&1\leq j\leq b\\
\Omega_n&=&|n|^2,&n\notin \{\n^{(1)},\cdots, \n^{(b)}\}
\end{array}\right.
\end{equation}
Equation \eqref{equ} is a well known model for the  natural NLS, in
which  the Fourier multiplier is substituted  by a multiplicative
potential $V$. Existence and stability of quasi--periodic solutions of \eqref{equ}
via a KAM algorithm was proved in \cite{EK} for the more general
case where $f(y)$ is substituted with $f(y,x)$, $x\in \T^d$. With
respect to that paper we use a different approach to prove measure estimates, based essentially on two ingredients:  the fact that  the equation has
the total momentum $M=\int_{\T^d}\bar u \nabla u $ as an integral of
motion, and the use of the properties of the quasi-T\"oplitz functions.
 These two ideas induce some significant simplifications which we think
are interesting, in particular the conservation of momentum   enables us to prove a stronger result, namely our solutions are analytic while in \cite{EK} only Gevrey class is proven.
Our dynamical result for the NLS \eqref{equ} is
\begin{Theorem}\label{quasiperiodic solution}
 %For any fixed
%$0< \gamma\ll 1$,\marginpar{what is the meaning of $\gamma$ in this theorem?}
There exists a positive-measure Cantor set $\mathcal C$ such that
for any $\xi=(\xi_1,\cdots,\xi_b)\in\mathcal C$, the nonlinear
Schr\"{o}dinger equation $(\ref{equ})$ admits  small amplitude  analytic
 quasi-periodic solutions. The solutions are linearly stable and  we give a reducible normal form close to them.
\end{Theorem}

This is obtained by proving that the NLS Hamiltonian fits the hypotheses of an abstract KAM theorem, see Theorem \ref{KAM}.

 Before describing our results and techniques more in detail, let us
make a very brief excursus on the literature on quasi--periodic solutions
for PDEs on $\T^d$ and on the general strategy of a KAM algorithm.

The existence of quasi--periodic solutions for equation \eqref{equ}
 (as well as for the non--linear wave equation) was first proved  by
  Bourgain, see \cite{B3} and \cite{B4}, by applying a combination of
   Lyapunov-Schmidt reduction and Nash--Moser generalized implicit
   function theorem in order to solve the small divisor problem.
   This method is very flexible and may be effectively applied in
   various contexts, for instance in the case where $f(y)$ has only
   finite regularity, see \cite{BB1} and \cite{BB2}. As a drawback
   this method only establishes existence of the solutions but does not give  information on the linear  stability.
In order to achieve this stronger result it is natural to extend to \eqref{equ}   the, by now classical, KAM techniques which were developed to study equation \eqref{equ}
 with Dirichlet boundary conditions on the segment $[0,\pi]$.
A fundamental hypothesis in the aforementioned  algorithms is that the eigenvalues $\Omega_n$ are simple,
and this is clearly not satisfied already in the case of equation \eqref{equ} on $\T^1$, where the eigenvalues are double.
%This hypothesis is required to impose the strong {\em Melnikov non--resonance conditions} which are a key point in performing the algorithm.
We mention that this hypothesis was weakened for the non--linear wave equation by Chierchia and You in \cite{CY}, by only requiring that the eigenvalues have finite and uniformly bounded multiplicity.  Their method however does not extend trivially to the NLS on $\T^1$ and surely may not be applied to the NLS in higher dimension, where the multiplicity of $\Ome_n$ is of order $\Ome_n^{(d-1)/2}$.   The first result on KAM theory on the torus $\T^d$ was given in \cite{GY3} for  the {\em non--local} NLS:
$$   iu_t-\tr u+{\mathbb M}_\xi
u+f(|\Psi_s(u)|^2)\Psi_s(u)=0,\qquad x\in \T^d,\ t\in \R,
$$
where $\Psi_s$ is a linear operator, diagonal in the Fourier basis and such that $\Psi_s(e^{\ii\langle n,x\rangle})= |n|^{-2s} e^{\ii\langle n,x\rangle} $ for some $s>0$. The key points of that paper are: 1. the  use of the conservation of the total momentum to avoid the problems arising from the multiplicity of the $\Ome_n$ and  2. the fact that the presence of the non--local operator $\Psi_s$ simplifies the proof of the {\em Melnikov non--resonance conditions} throughout the KAM algorithm.
As we mentioned before the more complicated problem of a KAM algorithm for the local NLS without momentum conservation was solved by Eliasson and Kuksin in \cite{EK}.

Let us briefly describe the general strategy in the KAM algorithm for equation \eqref{equ}.

We expand the solution in Fourier series as $u=\sum\limits_{n\in \Z^d} u_n\phi_n(x)$, here $\phi_n(x)=\sqrt{1\over {(2\pi)^d}}e^{i\langle
n,x\rangle}$ with $ n\in\Z^d $ is the standard Fourier basis.Then we introduce standard action-angle coordinates for the modes $\mathtt n_j$  by setting
$u_{\n_j}=\sqrt{I^{(0)}_j+I_j}e^{i\teta_j},j=1,\cdots,b$,  where the $I^{(0)}_j$ are arbitrary sufficiently small numbers.  Finally we set  $u_n=z_{n}= z_n^+$,  $\bar u_n= \bar z_n= z_n^-$ for all $n\neq \{\n^{(1)},\cdots, \n^{(b)}\}$.
We get
\begin{equation}\label{equ-h}H=\sum\limits_{1\leq j\leq b}\omega_j(\xi) I_j+\sum\limits_{n\in \Z^d_1}\Omega_n|z_n|^2+P(I,\teta,z,\bar z),\quad \Z^d_1:=\Z^d\setminus \{\n_1,\dots,\n_b\}.\end{equation} It is easily seen that $H$ and hence  $P$ preserve the total momentum (see  formula \eqref{mome} below)   moreover $P$ (and $\sum(\Omega_m-|m|^2) z_m\bar z_m$) are
T\"{o}plitz/anti-T\"oplitz functions, namely the Hessian matrix  $\partial_{z_m^\sigma}\partial_{z_n^{\sigma'}} P $ depends  on ${z_m^\sigma},z_n^{\sigma'} $only through $\sigma m+\sigma' n$.

Informally speaking the KAM algorithm consists in constructing a convergent sequence of symplectic transformations $ \Phi_\nu$ such that
\begin{equation}\label{porto} \Phi_\nu\circ  H :=H_\nu= \sum\limits_{1\leq j\leq b}\omega^{(\nu)}_j (\xi)I_j+\sum\limits_{n\in \Z^d_1}\Omega^{(\nu)}_n(\xi)|z_n|^2+P_\nu(\xi,I,\teta,z,\bar z), \end{equation}
where $P_\nu\to 0$ in some appropriate norm.  The symplectic transformation is well defined for all $\xi$ which satisfy the {\em Melnikov non--resonance conditions}:

\begin{equation}\label{melnikov}|\langle \ome^{(\nu)},k\rangle +\Ome^{(\nu)} \cdot l| \geq \gamma K_{\nu}^{-\varrho} \,,\;\end{equation} for all  $ k\in \Z^b, \;l\in \Z^{\Z_1^d}$ such that $(k,l)\neq (0,0)\,,\; |l|\leq 2$ and $|k|< K_\nu$. Here $\varrho,\gamma$ are  appropriate constants.
With these conditions in mind it is clear that a degeneracy
$\Ome^{(\nu)}_n= \Ome^{(\nu)}_m$ poses problems since the left
hand side in \eqref{melnikov} is identically zero for
$k=0,l=e_m-e_n$ ($e_m$ with $m\in \Z_1^d$ is the standard basis
vector).  To avoid this problem we use the fact that all the $H_\nu$
have $M$ as constant of motion. This in turn implies that some of
the Fourier coefficients of $P_\nu$ are identically zero so that the
conditions \eqref{melnikov} need to  be imposed only  on those $k,l$
such that  $\sum_{i=1}^b \n_{i}k_i + \sum_{m\in \Z_1^d} m l_m=0$.
Then, in our example,  $k=0$  automatically implies $n=m$. This is
the key argument used in \cite{GY3}. However, once that one has
proved that the left hand side of \eqref{melnikov} is never
identically zero, one still has to show that the quantitative bounds
of  \eqref{melnikov} may be imposed on some positive measure set of
parameters $\xi$. This is an easy task when $|l|=0,1$  or $l= e_m +
e_n $ but  may pose serious problems in the case $l= e_m-e_n$ where
the non--resonance condition is of the form
\begin{equation}\label{semel}
 |\langle \ome^{(\nu)},k\rangle +\Ome^{(\nu)}_m-\Ome^{(\nu)}_n|\geq \gamma K_{\nu}^{-\varrho} \,, \; \forall k\in \Z^b\,, \;n,m\in \Z_1^d: \; |k|< K_\nu
\end{equation} where $n-m= \sum_{i=1}^b \n_i k_i$.  Indeed in this case for every fixed value of $k$ one should in principle impose infinitely many conditions, since the momentum conservation only fixes $n-m$.    In \cite{GY3}, the presence of $\Psi_s$ implies that  $\Omega^{(\nu)}_m-|m|^2\approx {\varepsilon\over |m|^s}$ so that if $|m|^s > c |k|^\tau$ the variation of $\Omega^{(\nu)}$ is negligible. This implies in turn that one has to impose only finitely many conditions for each $k$.  In the case of equation \eqref{equ} however $s=0$, so that this argument may not be applied. One wishes to impose the non resonance conditions by verifying only a finite number of bounds for each $k$. To do this  one needs some control on  $\Omega^{(\nu)}_m-|m|^2$, for $|m|$ large,  throughout the KAM algorithm. The ideal setting  is when $\Omega^{(\nu)}_m-|m|^2$ is $m$--independent. This holds true for  the first  step of the KAM algorithm due to the fact that $P$
 is a T\"oplitz function. However it is easily seen that already $P_{1}$ is not a
T\"oplitz function and some wider class of functions must be defined.

  In order to control the shift of the normal frequency Eliasson and Kuksin in \cite{EK}  define   a {\em  T\"oplitz-Lipschitz property}, which  they show  is satisfied by the NLS Hamiltonian and  preserved through the KAM
iteration. With this property, they prove the existence of
KAM tori.  As a further difficulty they consider an NLS equation which does
not have $M$ as a constant of motion. This implies that some of the
Melnikov non--resonance conditions \eqref{semel} may not be imposed.
At each step of the KAM algorithm they thus obtain a more complicated normal form.

In order to describe the T\"{o}plitz-Lipschitz property, given an analytic function $A(z,\bar z)$, let
 $A^n_m(\pm)= \partial_{z_m}\partial_{z_n^{\pm }} A$ be its  Hessian matrix.
 For all
$n,m,c\in \Z^d$, one requres that the limit $A^{n}_m (\pm,c):=\lim\limits_{t\rightarrow
\infty} A^{n\mp t c}_{m+tc}(\pm)$ exists and  is attained with speed of order ${1\over t}$. In dimension $d>2$ one also requires similar conditions on the  limits
$\lim\limits_{s\rightarrow
\infty}A^{n\mp s c'}_{m+ s c'} (\pm,c)$   with $c'$ orthogonal to $c$.
  In \cite{GXY} an understanding of this
property in $\T^2$ is given. A key step is to divide the region
$\{|n-m|\leq N\}\subset \Z^d\times \Z^d$ in a finite number of Lipschitz domains.

In our paper we use a similar --but in our opinion more natural-- approach.
We define a class of functions, the {\em quasi--T\"oplitz functions} whose main properties are:
\begin{enumerate}
\item
the Poisson bracket of two quasi-T\"oplitz functions is quasi-T\"oplitz (Proposition \ref{main}),
\item the Hamiltonian flow generated by
a quasi-T\"oplitz function preserves the quasi-T\"oplitz property
 (Proposition \ref{main}),
 \item
the solution of the homological equation with a quasi-T\"oplitz  perturbation
is quasi-T\"oplitz (Proposition \ref{submain}).
\end{enumerate}
   Note that the  T\"{o}plitz-Lipschitz property of \cite{EK} is closed only with respect to Poisson brackets when one of the functions is quadratic, this makes our definitions more flexible.

    In this paper we strongly rely on the conservation of momentum for our definitions, however this condition is not necessary in order to define the quasi-T\"oplitz functions, see for instance
\cite{BBP1}. In the next paragraph we give a brief informal description of our method.

\subsection{Brief description of the strategy}
We start by fixing two diophantine exponents $\tau_0\ll \tau_1$. All our definitions and constructions are based on some parameters $N\gg 1$,  $\frac12 <\tet1,\mu <4$  and $ \tau_0\leq \tau\leq \tau_1/4d$ which are needed in order to ensure that the quasi-T\"oplitz functions are closed with respect to Poisson brackets (with slightly different parameters).

The first step in our construction is an intrinsic (and unique)  description of affine subspaces described by equations with integer coefficients.  We consider the equations $ v_i \cdot x =p_i$,  $i=1,\dots,\ell$ $x,v_i\in \Z^d$, $p_i\in \Z$ describing   the set of integral points $x$ in an affine subspace,  we then denote this set  by $[v_i;p_i]_\ell$ and, by abuse of notation, call it an affine subspace.
Given  $N\gg 1$, an $N$--optimal presentation of an affine subspace of codimension $\ell$ is a (uniquely fixed if it exists) list  $[v_i;p_i]_\ell$ such that the $|v_i|< C_1 N$ and the
$p_i$ are {\em positive, ordered} and {\em as small as possible} (see Definition \ref{N-opt}).

This decomposition holds also for a single point (when $\ell=d$, in this case an $N$--optimal presentation will surely exist). Then we use  the parameters $\frac12 <\tet1,\mu<4$, $\tau_0\leq\tau\leq  \tau_1/4d$ to define the  notion of {\em $\ell$--cut} for a  point  $m$ and of  {\em good points} of an affine subspace with respect to the parameters $(N,\tet1,\mu,\tau)$. Namely,  if $[v_i;p_i]_d$ is the $N$--optimal presentation of $m$, then $m$ has a cut at $\ell$ if $p_\ell <\mu N^\tau$ and $p_{\ell+1}> \theta N^{4d\tau}$. In the same way the $(N,\tet1,\mu,\tau)$--good points of an affine subspace $[v_i;p_i]_\ell$, with $p_\ell< \mu N^\tau$  are those points  of $[v_i;p_i]_\ell$ which have a cut at $\ell$  with   parameters $(N,\tet1,\mu,\tau)$ (see Definition \ref{cut}).

We then define the {\em $(N,\tet1,\mu,\tau)$--bilinear  functions}, i.e.    functions which are bilinear in the {\em high variables} $z_m^\s,z_n^{\s'}$ such that $|m|,|n| > \theta N^\td$ and both $m$ and $n$ have a cut with parameters $(N,\tet1,\mu,\tau)$. These functions may depend on  $I,\teta$ and on the   {\em small  variables} $z_j^\s$ with $|j|<\mu N^3$ in a possibly complicated way (see Definition \ref{taubilinear} for a precise statement).

Finally we define  the {\em piecewise T\"oplitz} functions as those   $(N,\tet1,\mu,\tau)$--bilinear functions which are T\"oplitz when restricted to  the $(N,\tet1,\mu,\tau)$--{\em good points} of any affine subspace (see Definition \ref{pitop} and Remark \ref{pollen}).

We can now define the  $(K,\tet1,\mu)$--quasi--T\"oplitz
functions.  Informally speaking given a function    $f$, for all $N>K, \tau_0 \leq \tau\leq \tau_1/4d$, we project it   on the   $(N,\tet1,\mu,\tau)$--bilinear functions   and  we say that $f$ is quasi-T\"oplitz if all these projections are {\em well approximated} by a piecewise T\"oplitz function.
To be more precise, $\tau$ controls the size of the error function, namely    the $(N,\tet1,\mu,\tau)$--bilinear  part of $f$ is approximated by a piecewise T\"oplitz function with an error of the order $N^{-4d \tau}$,  for all $N\geq K$ (see Formula \eqref{limi} and Definition \ref{topbis})

The role of the parameters $K,\tet1,\mu$ is to ensure  that if $f,g$ are quasi-T\"oplitz with parameters $K,\tet1,\mu$ then $\{f,g\}$ is  quasi-T\"oplitz  for all $\theta'>\theta$ and $ \mu'<\mu$ provided  $K'>K$ is large enough (see Proposition \ref{main}).

We  proceed by induction supposing that we have been able to perform $\nu$ KAM iterative steps and that we have a Hamiltonian of the form \eqref{porto} where $\sum_m (\Omega_m^{(\nu)}-|m|^2) |z_{m}|^2$ is quasi-T\"oplitz  with parameters $(K_\nu,\theta_\nu,\mu_\nu)$  (note that $K_\nu$ is the
ultra--violet  cut-off at step $\nu$).   In order to solve the homological equation (and hence pass to step $\nu+1$ ) we restrict to the subset of $\xi$ for which \eqref{melnikov} holds for all $k,m,n$ (satisfying momentum conservation) for some $\varrho:= \varrho(k,m,n)< 2d\tau_1$.  The main point is to show that this restriction on the parameters only removes a small measure set.

For all natural $N\geq K_\nu $  we introduce a decomposition of $\Z_1^d$ as

 \begin{equation}\label{pustola} \Z_1^d:=A_0
 \cup\Biggl(\bigcup_{\ell=1}^{d-1} A_\ell \Biggl) \cup \; \{|m|\leq 4N^\td\}  \,,% \begin{equation}\Z_1^b\cap \{|m|\geq 4N^\td\}\subseteq \bigcup\limits_{0\leq \ell<d,i\leq\ell\atop |v_i|\lessdot N,p_i\in \Z}[v_i;p_i]^g_\ell
\end{equation}
here $A_0\equiv A_0(N)$ is  $\Z_1^d$ minus a finite number of affine hyperplanes while $A_\ell:= A_\ell(N)$ is the union of a finite number of affine spaces of codimension $\ell$ minus a finite number of affine spaces of codimension $\ell+1$ (see Figure \ref{fig1} for a picture in $d=2$).

 This  decomposition is constructed as follows:

$A_0$ (defined in formula \eqref{nodiv}) is chosen so  that for all $|k|< N, m\in A_0$ the Melnikov denominators \eqref{semel} are {\em not small}.

For all $0<\ell<d$ we  may write

$$\quad A_\ell:= \bigcup\limits_{ v_1,\dots,v_\ell \in \Z_1^d\,, p_1,\dots, p_\ell \in \Z\atop
 |v_i|< C_1N \,, p_i< 4 N^{\tau_1/4d}}  [v_i;p_i]_\ell^g\,,
$$
where  the  $[v_i;p_i]_\ell^g\subset [v_i;p_i]_\ell$ (see Definition \ref{papilla})  are defined in order to ensure the following property:
fix $\tau(p_\ell)$ by setting $N^\tau =\max( 2p_\ell ,N^{\tau_0})$,  we have that all $m\in [v_i;p_i]_\ell^g$ are  $(N,\tet1,\mu,\tau(p_\ell))$--good points for  $[v_i;p_i]_\ell$    for all choices of $\frac12<\tet1,\mu<4$ -- this is the content of Lemma \ref{minko}. Finally the fact that this sets provide a decomposition of $\Z^d_1$ is the content of Proposition \ref{key2}.

 To prove the measure estimates we  use the above decomposition with $N= K_\nu$. Then   the quasi-T\"oplitz  property with $ N= K_\nu$ implies that
 for each $m\in[v_i;p_i]^g_\ell$,
\begin{equation}\label{stima}\Ome_m^{(\nu)}= |m|^2 + {\it \hat\Ome}^{(\nu)}([v_i;p_i]_\ell)+ \bar\Ome_m^{(\nu)} K_\nu^{-4d\t1(p_\ell)},\end{equation} where ${\it \hat\Ome}^{(\nu)}$ is constant on all the points of $ [v_i;p_i]_\ell$ while $\bar\Ome_m^{(\nu)}$ is bounded by $\varepsilon_0$ (see Lemma \ref{diago}).
We stress   that here\footnote{Note that in the definition of quasi-T\"oplitz functions and of cuts, instead, $\t1$ is left as a free parameter with the only restriction $\tau_0\leq \t1\leq \td/4d$.} $\tau=\tau(p_\ell)$ is fixed by   the positive integer $p_\ell$.

Roughly speaking, we  fix $k$,  choose {\bf one point} $m^g$  on each  $[v_i;p_i]_\ell^g$ and impose the Melnikov conditions \eqref{semel} with $\varrho= 2 d\tau(p_\ell) $, $ \gamma \rightsquigarrow 2\gamma$  and $m=m^g$ (see Definition \ref{oppio} {\it iv)} for the precise formulation).  This condition and \eqref{stima} ensure the second Melnikov condition for all $m \in [v_i;p_i]_\ell^g$ with $\varrho=2d\tau(p_\ell)$ (see Lemma \ref{key}).
This shows that the infinitely many conditions \eqref{semel} can be imposed by only requiring a finite subset of them.

In order to check the measure estimates we remark that to impose {\bf one} Melnikov condition (i.e. with fixed $k$, $m\in [v_i;p_i]_\ell^g $ and $\varrho= 2 d\tau(p_\ell)$)  we need to remove a region of parameter sets of order $K^{-2 d\tau(p_\ell)}_\nu $(see Lemma \ref{misu}). Thus we need to estimate the number of affine spaces $[v_i;p_i]_\ell$ with $p_\ell=p$,  using Remark \ref{numero} it follows that this bound is proportional to $K_\nu^{d \t1(p)}= (2p)^d $. This concludes the problem of measure estimates and we exclude a set of $\xi$ of measure $\sum_{p\in \N : p> K_\nu^{\tau_0}}(2p)^{-d}$ (here we are giving only an informal argument,   see Lemma \ref{measure} for the complete proof).
In order to pass to the step $\nu+1$ we need $F_\nu$ (the solution of the Homological equation) to be quasi-T\"oplitz: this requires a further restriction of the parameter set (see Definition \ref{oppio}{\it iv)}, Remark  \ref{quiri} and Proposition \ref{submain}).

 Recalling that quasi-T\"oplitz functions are closed with respect to Poisson brackets we conclude   that the new Hamiltonian is still quasi--T\"oplitz  for some new parameters $\theta_{\nu+1},\mu_{\nu+1}$ for all $N\geq K_{\nu+1}$.

\section{Relevant notations and definitions}

\subsection{ Function spaces and norms}
We start by introducing some notations. We fix $b$ vectors $\{\n^{(1)},\cdots, \n^{(b)}\}$ in
$\Z^d$  called  the {\em tangential sites}. We denote by $\Z^d_1:=\Z^d\setminus
\{\n^{(1)},\cdots, \n^{(b)}\}$ the complement, called the {\em normal sites}. Let $z=(\cdots, z_n,\cdots)_{n\in \Z^d_1}$,
and its complex conjugate $\bar z=(\cdots,\bar z_n,\cdots)_{n\in
\Z^d_1}$. We introduce the weighted norm
$$\|z\|_{\rho} =\sum_{{n\in\Z^d_1}}|z_n|e^{|n|\rho}|n|^{d+1},$$
where $|n|=\sqrt{n_1^2+n_2^2+\cdots+n_d^2}$,
$n=(n_1,n_2,\cdots,n_d)$ and $\rho
> 0$. We denote by $\ell_\rho$ the Hilbert space of lists $\{w_j=(z_j,\bar z_j)\}_{j\in \Z_1^d}$ with $\|z\|_\rho<\infty$.

 We consider the real torus  $\T^b:=\R^b/\Z^b$ naturally contained in the space   $\C^b/\Z^b\times\ell_\rho $  as the subset where $I=z=\bar z=0$.
 We then consider in this space the   neighborhood of $\T^b :$
$$D(r,s):=\{(I,\teta,z,\bar z):|{\rm Im}\, \teta|<s,|I|<r^2,{\|z\|}_{\rho}<r,
{\|\bar z\|}_{\rho}<r\},$$
 where $|\cdot|$ denotes the sup-norm of complex vectors. Denote by $\mathcal O$ an open and bounded parameter
set in $\R^b$ and let $D= \max_{\xi,\eta\in \mathcal O}|\xi-\eta|$.

 We consider   functions $F(I,\teta,z;\xi): D(r,s)\times {\mathcal O}\to \C$ analytic in $I,\teta,z$ and of class $C_W^1$ in $\xi$.  We expand in  Taylor--Fourier series as:
\beq\label{2.2}
 F(\teta, I, z, \bar z;\xi )=\sum_{l,k,\alpha,\beta }
F_{lk \alpha\beta }(\xi)I^le^{{\rm i}
 \la k,\teta\ra}z^{\alpha} \bar z^{\beta },\eeq
 where the coefficients $F_{lk \alpha\beta }(\xi)$ are of class $C^1_W$ (in the sense of Whitney), the vectors $\alpha\equiv (\cdots,\alpha_n,\cdots)_{n\in\Z_1^d}$, $\beta \equiv
(\cdots, \beta _n, \cdots)_{n\in\Z_1^d}$  have finitely many non-zero  components $\alpha_n,\beta_n\in \N$, $z^{\alpha} \bar z^{\beta }$ denotes $\prod_n
z_n^{\alpha_n}\bar z_n^{\beta_n}$ and finally $\la\cdot,\cdot\ra$ is the standard inner product in $\C^b$.

We use the following   weighted norm for $F$:
 \begin{equation}\label{2.3}
  \|F\|_{r,s}=\|F\|_{ D(r,s)  ,\mathcal O}\equiv \sup_{{\|z\|_{\rho}<r}\atop{\|\bar
z\|_{\rho}<r}}\sum_{\alpha,\beta,k,l } |F_{kl\alpha \beta
}|_{\mathcal O}\ r^{2|l|}e^{|k|s}\,|z^{\alpha}| |\bar z^{\beta }|,
\end{equation}

 \beq\label{2.4}
 |F_{kl\alpha \beta }|_{\mathcal O}\equiv \sup_{\xi\in \mathcal O}
 (|{F_{kl\alpha \beta }}|+
 |\frac{\partial F_{kl\alpha \beta }}{\partial \xi}|).\eeq
 \noindent (the derivatives with respect to $\xi$ are in the sense of Whitney).
 To an analytic function  $F$, we associate  a Hamiltonian vector
field with coordinates
$$ X_F=(F_I, -F_\teta, \{{\rm
i}F_{z_n}\}_{n\in \Z_1^d}, \{-{\rm i}F_{\bar z_n}\}_{n\in\Z_1^d}).$$

 Consider a vector  function $G: D(
r,s)\times {\mathcal O}\to \ell_\rho$ with $$G= \sum_{kl\alpha \beta } G_{kl\alpha \beta } (\xi)I^le^{{\rm i}
 \la k,\teta\ra}z^{\alpha} \bar z^{\beta}, $$  where $G_{kl\alpha \beta }=(\cdots,G_{kl\alpha \beta }^{(i)},\cdots)_{i\in \Z_1^d}$. Its norm is similarly defined
as $$\|G\|_{D( r,s), \mathcal O}=\sup_{{\|z\|_{\rho}<r}\atop{\|\bar
z\|_{\rho}<r}} \| \mathcal MG\|_{\rho}$$ where $$\mathcal MG=(\cdots,\mathcal MG^{(i)},\cdots)_{i\in \Z_1^d},\qquad \mathcal MG^{(i)}= \sum_{\alpha,\beta,k,l
} |G^{(i)}_{kl\alpha \beta }|_{\mathcal O}\
r^{2|l|}e^{|k|s}\,z^{\alpha}\bar z^{\beta }$$ is a majorant of
$G^{(i)}$. We say that  an analytic function $F$ is regular if the
function $(z,\bar z)\to \mathcal M X_F$ is analytic from $B_r\to
 \ell_\rho$. Its
 weighted  norm is defined by\footnote{ The norm  $\|\cdot\|_{D_{\rho}(r,s), \mathcal O}$ for scalar functions is defined in
(\ref{2.3}). }  \begin{eqnarray}\label{norvec}
\|X_F\|_{r,s}=\|X_F\|_{\!{}_{ D(r,s), \mathcal O}}&\equiv&
\sum_{j=1}^b\|F_{I_j}\|_{\!{}_{ D(r,s), \mathcal O}}+ \frac
1{r^2}\sum_{j=1}^b \|F_{\teta_j}\|_{\!{}_{ D(r,s), \mathcal O}}\nonumber\\
&+& \frac 1r(\|\partial_{z}F\|_{ D(r,s),\mathcal O}+\|\partial_{\bar z}F\|_{ D(r,s),
\mathcal O}). \label{2.6}
\end{eqnarray}

 A function $F$ is said to satisfy momentum conservation if $\{F,M\}=0$ with $M= \sum_{i=1}^b \n^{(i)} I_i+\sum_{m\in \Z_1^d} j |z_m|^2$. This implies that
 \begin{equation}\label{mome}
F_{k,l,\alpha,\beta}=0\,,\qquad {\rm  if }\; \;\pi(k,\alpha,\beta):= \sum_{i=1}^b  \n^{(i)} k_i +\sum_{m\in \Z_1^d} m (\alpha_m-\beta_m) \neq 0.
\end{equation}
By Jacobi's identity momentum conservation is preserved by Poisson bracket.
\begin{remark}
 It will be useful to envision the conservation of momentum at fixed $k$ as a relation between $\a,\b$; to make this more evident we write
 \begin{equation}\label{patacc}
 \pi(k,\alpha,\beta)=0 \,,\quad {\rm as}\quad  -\sum_{m\in \Z_1^d} m (\alpha_m-\beta_m)= \sum_{i=1}^b  \n^{(i)} k_i:= \pi(k)
\end{equation}
  \end{remark}
\begin{definition}
We denote by  ${\mathcal A}_{r,s}$ the space of regular analytic
functions in $D(r,s)$ and $C^1_W$ in $\mathcal O$ which satisfy
momentum conservation \eqref{mome} and with finite semi-norm \eqref{norvec}
\end{definition}
If $\mathcal S$ is a set of monomials in $I_j,e^{\ii\teta_j},z_m,\bar z_n$, we define the projection operator $\Pi_{\mathcal S}$  which to a given analytic function $F$ associates  the part of the series only relative to the monomials in $\mathcal S$.

We have following useful result
\begin{Lemma}\label{cauchy} $i)$ The majorant norm is closed under projections, namely  $\|\Pi_{\mathcal S}f\|_{r,s}\leq \|f\|_{r,s}$,  and $\|X_{\Pi_{\mathcal S}f}\|_{r,s}\leq \|X_f\|_{r,s}$ .\\
$ii)$ ${\mathcal A}_{r,s}$ is closed under Poisson brackets, with respect to the symplectic
form $ dI\wedge d\teta+ i d z\wedge d\bar z$, moreover by Cauchy
estimates, if we denote $\delta= (\frac{r'}r)^2 \min( s-s',1-\frac rr')$,
$$ \|[X_f,X_g]\|_{ r',s'}\leq  2^{2d+1}\del^{-1} \|X_f\|_{ r,s}\|X_g\|_{r,s}\,,$$
$$ \|X_{\{f,g\}}\|_{r',s'}\leq  2^{2d+1} \delta^{-1}\|X_f\|_{ r,s}\|X_g\|_{r,s} \,,\qquad$$
\end{Lemma}
\begin{proof}Item {\it i)} is obvious. Item   {\it ii)}  is proved in \cite{BBP}, respectively Lemmata  2.15 and 2.16. In  \cite{BBP} the interested reader can find an analysis of the properties of the majorant norm. Note that in  \cite{BBP} there is the restriction $r/2<r'<r$ (same for $s$) hence the term $(\frac{r}{r'})^2$ is substituted by $4$.
\end{proof}
\section{Affine subspaces}

An affine space $A$ of codimension $\ell$ in $\R^d$  can be defined
by a list of $\ell$ equations $A:=\{x\,|\,v_i\cdot x=p_i\}$ where
the $v_i$ are independent row vectors in $\R^d$.  We will write shortly that
$A=[v_i;p_i]_{\ell}$. We will be interested in particular in the
case when $v_i,p_i$ have integer coordinates, i.e.  are {\em integer
vectors} and the vectors $v_i$  lie in a prescribed ball $\Bn$ of radius  some constant $N$.
%\marginpar{non vedo l'utilit\`a di restringersi a $\Z^d_1$  a questo punto}
We set $C_1:= \max_i |\mathfrak n_i|$, and we denote by $$\langle v_i\rangle_{\ell}={\rm
Span}(v_1,\dots,v_\ell;\R)\cap \Z^d\,,\quad \Bn:=\{x\in
\Z^d\setminus \{0\}\,:\; |x| < C_1N\}  ,$$  here $N$ is
any large number.  In particular we  implicitly assume that $\Bn$  contains a basis of $\R^d$. \smallskip

For given $s\in \N$, in the set of vectors $\Z^s$  we can define  the {\em sign lexicographical order} as follows.\begin{definition}
 Given $a=(a_1,\ldots,a_s)$ set $(|a|):=(  |a_1|,\ldots,|a_s|)$ then we set $a\prec b$  if either $(|a|)< (|b|)$ in the lexicographical \footnote{ Recall  that given two partially ordered sets $A$ and $B$, the lexicographical order on the Cartesian product $A \times B$ is defined as
$(a,b) <  (a',b')$ if and only if  either $a < a'$  or $a = a'$ and $b <  b'$.}  order  (in $\N^s$) or if  $(|a|)=(|b|)$ and $a>b$ in the lexicographical order in $\Z^s$.
\end{definition}
 For instance in $\Z^2$, $(\pm 1,\pm 5)\prec (\pm 2, \pm 4)$ since $(1,5)< (2,4)$; on the other hand we have $(1,4)\prec (1,-4) \prec (-1,4) \prec (-1,-4)$. This is due to the fact that these last vectors have the same components apart from the sign and $(1,4)> (1,-4)>(-1,4) > (-1,-4)$ in the lexicographic ordering of $\Z^2$.
  \begin{Lemma}
Every non empty set of elements in $L\subset \Z^s$ has a unique minimum.
\end{Lemma}
 \begin{proof}
 We first consider the list of vectors $|L|\subset \N^s$ consisting of the vectors $(|a|)$ with $a\in L$. This list has a minimum with respect to the lexicographic ordering of $\N^s$.  Naturally there may more than one vector, say  $a\neq b\in L$ with   $(|a|)=(|b|)$,  which attain the minimum of $|L|$.  This vectors are at most $2^s$ and among them we choose the unique maximum in the lexicographical order in $\Z^s$.
\end{proof}

 Consider a fixed but  large enough $N$.
  \begin{definition}
We set  $\mathcal H_{N }$ the set of all   affine spaces $A$ which can be presented as $A=[v_i;p_i]_\ell$  for some $0<\ell\leq d$ so that that $v_i \in \Bn$.
\end{definition}We display as  $(p_1,\ldots,p_\ell;v_1,\ldots,v_\ell)$  a given presentation, so that it is a vector in $\Z^{\ell(d+1)}$. Then we can say that $[v_i;p_i]_\ell \prec [w_i;q_i]_\ell$ if  $(p_1,\ldots,p_\ell;v_1,\ldots,v_\ell)\prec (q_1,\ldots,q_\ell;w_1,\ldots,w_\ell)$.
\begin{definition}\label{N-opt}
 The $N$--optimal presentation $[l_i;q_i]_\ell$ of $A\in \mathcal H_{N }$ is the minimum in the sign lexicographical order of the presentations of $A$ which satisfy the  bound $v_i\in B_N$.

 Given an affine subspace $A:=\{x\,\vert v_i\cdot x =p_i\,,\  i=1,\dots,\ell\}$ by the notation $A\frec{N} [v_i;p_i]_\ell$ we mean that the given presentation is $N$ optimal.
 \end{definition}
\begin{remark}\label{lordi} i) Note that each point $m=(m_1,\ldots,m_d)\in \Z^d_1$ has a $N$--optimal presentation (this presentation is usually not the naive one $[e_i,m_i]_d$ where the $e_i$ form the standard  basis of $\Z^d$).

ii) We may use the ordering given by  $N$ optimal presentations of points  in order to define a new lexicographic order on $\Z^d$ which we shall denote by $a\prec_N b$ or $a\prec b$ when $N$ is understood.
\end{remark}
\begin{Example}\label{ex1}{\rm  We now give an example of the $N$--optimal presentation of a point and of an affine subspace. One may easily verify that for any affine subspace $A$ there exists $\bar N(A)$ such that for all $N\geq \bar N(A)$ the $N$--optimal presentation is $N$ independent.

Let us start with the case $m_0= (-11,15,3,27)\in \Z^4$.  We have that $\forall  N> C_1^{-1} \sqrt{82}$ (recall that $C_1= \max_i|\mathtt n_i|$)
$$ m_0 \frec{N}[0,0,0,1; (0,0,9,-1), (0,1,4,-1),(3,0,2,1), (1,0,-5,1)]\,.$$
In general given any point $m_0$ we will always find $\bar N(m_0)$ such that for all $N> \bar N(m_0)$ the $N$--optimal presentation is fixed say $[p_i^{(0)};v_i^{(0)}]_d$ and $p_i^{(0)}=0$ for $i=1,\dots,d-1$ while $p^{(0)}_d=$ mcd $(m^{(0)}_1,\dots,m_d^{(0)})$.
%C_1onsider a point $m_0\in \Z^d$ since by definition $m_0= \sum_{i=1}^d m^{(0)}_i {\bf e}_i$ with $|{\bf e}_i|=1$ we have that $[m^{(0)}_i; {\bf e}_i]$ is a presentation of the space so $m_0\in \mathcal H_N$. Let $v_1,\dots,v_{d-1}$ be a basis of integral vectors orthogonal to $m_0$ and fix $N\geq C_1 \max (|v_i|)$, finally let $v_d$ be a vector that realizes $(m_0,v_d)=$ mcd$(m^{(0)}_i)_{i=1}^d$. Then
%$$ m_0\frec{N}[0,\dots,0, {\rm mcd}(m^{(0)}_i)_{i=1}^d;v_1,\dots,v_d] $$ is an optimal presentation of $m_0$ provided we choose the $v_i$ as small as possible.

Let us now study some affine subspaces.

If $d=2$ consider  the line $A:=\{m\in \Z^2:\; m= m_0+ t c\,,\; t\in \R\}\,,$ with $m_0$  orthogonal to $c$ (suppose also that the components of $m_0$ are coprime). Then
$A\frec{N}[|m_0|^2; m^{(0)}]_1$ provided that $N\geq C_1^{-1}|m_0|$.

 If $d=4$ and  $ A:= \{m\in \Z^4:\, m=(-11,15,3,27)+ (1,0,0,0)t \,,\quad t\in \R\} $ we have that
 $$ A\frec{N} [0,0, 3; (0,0,9,-1), (0,1,4,-1),(0,0,1,0))]_3\,,\quad \forall N> C_1^{-1} \sqrt{82}\,.$$

 $${\rm If}\;  B:= \{m\in \Z^4:\, m=(-11,15,3,27)+ (1,0,0,0)t +(0,1,0,0)s \,,\quad t,s\in \R\} $$ we have that
 $$B\frec{N}[0,3;  (0,0,9,-1), (0,0,1,0)]_2\quad \forall N> C_1^{-1} \sqrt{82}$$
%
%More in general if $$A:=\{x\in \Z^d:\; x= m_0+ t c\,,\; t\in \R\}\,,$$ where $m_0,c\in \Z^d$,  $m_0$ is orthogonal to $c$ and $m_0,c$ both have coprime components (so that for instance $m_0$ is  the smallest integral vector on the line $s m_0$, with $s\in \R$), we proceed as follows.
%
% In the case $d=2$  the optimal presentation is $[|m_0|^2; m^{(0)}]_1$ provided that $N\geq |m_0|$. In general dimension for all $N\geq \bar N(m_0,c)$ large enough we have the fixed $N-$optimal presentation for the line given by $$A\frec{N}[0,\dots,0,p^{(0)}_{d-1};v^{(0)}_1,\dots,v^{(0)}_{d-2},v^{(0)}_{d-1}]_{d-1}$$ where $v^{(0)}_1,\dots,v^{(0)}_{d-2}$ are orthogonal to both $c,m_0$ while $v^{(0)}_{d-1}$ is orthogonal to $c$ but not to $m_0$.
 }

\end{Example}
\begin{Lemma}
i) If the presentation $A=[v_i;p_i]_\ell$ is $N$--optimal,  we have
\begin{equation}\label{va2}
0\leq  p_1 \leq p_2 \leq\ldots\leq  p_\ell   \end{equation}

 ii) For
all $j<\ell$ and for which $v\in \langle v_1,\dots,v_\ell\rangle\cap \Bn\setminus\langle
v_1,\dots,v_j\rangle$, one has:
\begin{equation}\label{va3}
|(v,r)|\geq  p_{j+1}\,,\quad  \forall r\in A\end{equation}

iii) Given $j<\ell$ set $A_j:=\{x\,|\,v_i\cdot x=p_i,\ i\leq j\},$ then the presentation $A_j=[v_i,p_i]_j$ is $N$--optimal.

iv) Finally   $-A$  has a $N$--optimal presentation $-A =[v'_i,p_i]_\ell $ with the same constants $p_i$ and   $(|v'_i|)=(|v_i|)$.
\end{Lemma}
 \begin{proof}
i) If $p_i<0$  we can change the presentation changing $p_i$ into $-p_i$ and $v_i$ into $-v_i$. By definition this is a lower  presentation lexicographically, we obtain a contradiction.  Suppose now that \eqref{va2} is false -say for instance that $p_1> p_2\geq 0$-  then by definition $\{p_2,p_1,\dots p_\ell; v_2,v_1,\dots,v_\ell \}$ is a presentation of $A$ and it is lexicographically lower than $\{p_1,p_2,\dots p_\ell; v_1,v_2,\dots,v_\ell \}$.

ii) Take $v\in \langle v_1,\dots,v_\ell\rangle\cap \Bn\setminus\langle v_1,\dots,v_j\rangle$ and any $r\in A$. We note   that $(v,r)$ is constant on $A$.
There exists an $h>j$ such that if we substitute  $v_h,\, h>j,$ with $v$ we obtain a
new presentation. Again we deduce by minimality in the
lexicographical order, that $|(v,r)|\geq  p_h  \geq  p_{j+1}  $.
\smallskip

iii) Any presentation $A_j=[w_i,q_i]_j$ can be completed to a presentation $[w_i,q_i]_\ell$ of  $A$  so  if $[q_1,\ldots,q_j,w_1,\ldots,w_j]\prec [p_1,\ldots,p_j;v_1,\ldots,v_j]$ we also have   $[q_1,\ldots,q_\ell;w_1,\ldots,w_\ell]\prec [p_1,\ldots,p_\ell;v_1,\ldots,v_\ell]$ by the definition of lexicographical order, a contradiction.

 iv)
 As for the last statement it is enough to observe that there is a 1--1 correspondence between presentations  $A=[w_j,q_j]$ of $A$ and $-A$ with the constants $q_i\geq 0$, if $A=[w_j,q_j]$  we have $-A=[ -w_j ,q_j]$. The absolute value vectors of the two presentations are the same, the  statement follows.
 \end{proof}

\begin{remark}\label{numero}
For fixed $N$, $\ell$, $p$  the number of  affine spaces in $\mathcal H_{N }$ of codimension $\ell$ and such that $ p_\ell \leq p$ is bounded by
$(2C_1N)^{\ell d} (2p)^{\ell}$.
\end{remark}
\subsection{ Parameters and cuts}  We shall need several auxiliary parameters in the course of our proof.
 We start by fixing some numbers
    \begin{equation}
\label{itau}\tau_0> \max( d+ b,12),\quad \td:=(4d)^{d+1}(\tau_0+1)\,,
\end{equation}
$$
c\leq \frac12\,,\; C\geq 4 \,,\; N_0 \geq  d! C_1^{d} C c^{-1}.
$$
In what follows $N$ will always denote some large number, in particular $N>N_0$, for the purpose of this paper we may fix $c=\frac12$ and $C=4$, however we give the definitions in the more general setting so that they are more flexible. \smallskip

We assume that $N$ has been fixed.
  Given a point $m$ we write $m \frec{N}[v_i;p_i]$ for its optimal presentation  dropping
the index $\ell$ which for a point is always $\ell=d$.
 Set by
convention $p_{0}=0$ and $p_{d+1}=\infty$.

We  then give a definition involving the parameters $ \tet1,\mu,\tau$ which we call {\em allowable} if
$$\tau_0\leq \tau\leq\tau_1/(4d)=  ( 4 d)^{ d }(\tau_0+1),\quad  c< \tet1,\mu<C.$$

We need to
analyze certain {\em cuts}, for the values $p_i$ associated to an optimal presentation of a point.  This will be an index $\ell$  where the values of the $p_i$ jump according to the following:

\begin{definition}\label{cut}
The point $m \frec{N}[v_i;p_i]$ has a {\em cut}  $\ell\in\{0,1,\dots,d\}$  with the  parameters  $(N, \tet1,\mu,\tau)$,   if $\ell$ is such that $  p_\ell  < \mu
N^{\tau}$, $ p_{\ell+1 }  > \theta N^{4 d \tau}$ (recall that $p_0=0,p_{d+1}=\infty$). \smallskip

The space $A:=\{x\,|\, v_i\cdot  x =p_i,\ i=1,\ldots,\ell\}$  is denoted by $[v_i;p_i]_\ell$  and called the {\em affine space associated to the cut} of $m$.

In turn  for every affine subspace $A\frec{N} [v_i;p_i]_\ell$ with $p_\ell < \mu N^\tau$, the set of points $m\in A$ with $|m|> \theta K^{\tau_1}$ which have $\ell$ as a cut with the  parameters $ (N,\tet1,\mu,\tau)$   are called  the $ (N,\tet1,\mu,\tau)$--good points of $A$.
\end{definition}

Notice that  $\theta N^{4 d \tau}>\mu
N^{\tau}$ (since $N^{(4d-1)\tau}\geq N^{(4d-1)\tau_0}>Cc^{-1}>\theta\mu^{-1}$), so  for any given $m\in \Z^d_1$   there is at  most {\bf one} choice of $\ell$ such that $m$ has a $\ell$ cut with parameters $(N,\tet1,\mu,\tau)$.
Note moreover that the affine subspace associated to a $ (N,\tet1,\mu,\tau)$--good point of $A$ is $A$.
\begin{remark}
The purpose of defining a cut $\ell$ is to separate the
numbers $p_i$ into {\em small} and {\em large}. The parameters $ (N,\tet1,\mu,\tau)$
give a quantitative meaning to this statement.
\end{remark}
\begin{Example}\label{ex2}
{\rm  Fix $N>C_1^{-1}\sqrt{82}$, $\theta,\mu,\tau$ and consider the affine subspace
  $ A\frec{N} [0,0, 3; (0,0,9,-1), (0,1,4,-1),(0,0,1,0))]_3$ of Example \ref{ex1}.  For all $t$ large enough (i.e. $t> 66 C_1 N$ ), setting
 $$m(t)=(-11,15,3,27)+ (1,0,0,0)t$$ we have $$m(t) \frec{N} [0,0, 3, p_4(N,t); (0,0,9,-1), (0,1,4,-1),(0,0,1,0),v_4(N))], $$ where
 $v_4(N)=(v^{(1)}_4(N), \dots, v^{(4)}_4(N))$ is a vector such that: $|v_4(N)|<C_1 N$, the first component  $v_4^{(1)}(N)=1$; finally  $p_4(N,t)= t- P(N)$ with $|P(N)|<33 C_1 N$.  Hence $m$ is
  a $(N,\theta,\mu,\tau)$ good point of  $A$ provided that $t> \theta  N^{4d\t1}- 33 C_1 N$.}
\end{Example}
 \begin{remark}\label{varlm}
1)\quad If    $\ell$   is a  cut  for the
point $m \frec{N}[v_i;p_i]$, with allowable parameters $ (N,\theta',\mu' ,\tau)$ it is also so for all parameters $(N,\tet1,\mu,\tau)$ with $ c<\theta\leq
\theta'<C,\ c<\mu'\leq \mu<C $.

2)\quad If for a given $\ell,\tau_0\leq \tau\leq \td/4d$ we have  $p_\ell\leq cN^\tau,\, p_{\ell+1}\geq CN^{4d\tau}$, then $\ell$ is a cut with parameters $(N, \tet1,\mu,\tau)$ for every choice of  $ c<\tet1,\mu<C$.

\end{remark}
\begin{Lemma}\label{mah}
Consider $m,r \in \Z^d_1$ with $m\frec{N}[v_i;p_i]$, $r\frec{N}[w_i;q_i]$
 suppose that $\ell$  is a  cut for  $m$ with the allowable parameters $ N,\theta',\mu',\tau$,  and suppose there exist   parameters   $ c<\theta<\theta'<C$, $c<\mu'<\mu<C $:
 \begin{equation}
\label{vicin}|r-m|<C_1^{-1}(\mu-\mu')N^{\tau-1},\ C_1^{-1}(\theta'-\theta)N^{4d\tau-1} .
\end{equation} then:

 (1) $\ell$   is a   cut  for the point $r$, for all  allowable parameters  $(N,\theta, \mu,\tau)$  for which \eqref{vicin} holds.
\smallskip

(2) $\langle
w_1,\dots,w_\ell\rangle=\langle v_1,\dots,v_\ell\rangle$.

(3)  $[w_i;q_i]_\ell=[v_i;p_i]_\ell+r-m$.
\end{Lemma}

  \begin{proof}  Fix $\tet1,\mu$ satisfying \eqref{vicin}.
Write $ (v_i,r)= (v_i, r-m)+p_i$. For $i\leq \ell$, since $|v_i|\leq C_1N$  we have:
\begin{equation} \label{va} |(v_i,r)|\leq  p_i +|v_i|| r-m|<
\mu' N^{\tau}+ (\mu-\mu')N^{\tau}= \mu N^{\tau}.
\end{equation} From Formula \eqref{va2}   by the definition of $N$--optimal, for all
$v\in \Bn\setminus\langle v_1,\dots,v_\ell\rangle$ one has
\begin{equation}\label{vab} \quad\quad|(v,r)|=|(v,m)+(v,r-m)| \geq
 p_{\ell+1 }  - | v| |r-m |>\theta' N^{4 d\tau}-C_1N |r-m|
= \theta N^{4 d\tau}.
\end{equation}

\noindent (1), (2)\quad By induction on $i$ we wish to show that $ q_i<\mu N^{\tau}$
and $w_i\in \langle v_1,\dots,v_\ell\rangle$ for all $i\leq \ell$. For $i=0$ this is trivial, so assume that
for $0\leq i<\ell$, we have $\langle w_1,\dots,w_i\rangle\subset
\langle v_1,\dots,v_\ell\rangle$. Since the $v_i$ are independent,
 there exists $h\leq \ell$ such that $v_h\notin \langle w_1,\dots,w_i\rangle$. By \eqref{va} $ q_{i+1}\leq |(v_h,r)| < \mu N^{\tau}.$

 By contradiction suppose that $w_{i+1}\in \Bn\setminus\langle
v_1,\dots,v_{\ell}\rangle$, applying  formula \eqref{vab} we would get $(w_{i+1},r):= q_{i+1}> \theta N^{4 d\tau}>\mu N^{\tau}$, a contradiction.

%Thus, since $\theta' N^{4d\tau}\geq \theta N^{4d\tau}>\mu N^{\tau}$, this also implies $w_{i+1}\in  \langle v_1,\dots,v_\ell\rangle$.\smallskip

 Since the $w_i$  (as well as the  $v_i$) are  linearly independent, clearly $\langle v_1,\dots,v_\ell\rangle=\langle
 w_1,\dots,w_\ell\rangle$.  This proves (2).
 As a consequence   for $s>\ell$, we apply again formula \eqref{vab} to
   $w_{s}\in \Bn\setminus\langle v_1,\dots,v_\ell\rangle $; we obtain $ q_{j+1}> \theta   N^{4 d\tau}$. This completes the proof of (1).

(3) \quad  By (2) the space $[w_i;q_i]_\ell$ is the one parallel to $[v_i;p_i]_\ell$ and passing through $r$. The result follows.\end{proof}

%\begin{remark}\label{mahremark} If  $  \theta',\mu',\tau$ are allowable parameters and  $ \theta<\theta',\mu'<\mu $  are such  that  $(\mu-\mu')N^{\tau},(\theta'-\theta)N^{4d\tau}>a N^4$ or
%$$\mu>\mu'+ aN^{4 - \tau } ,\quad \theta'- aN^{4(1-d\tau)}>\theta $$\marginpar{la formula era sbagliata, ho un po cambiato tutto}then, by Remark \ref{alpara}, $N,\mu ,\theta ,\tau$ are allowable parameters if  $ \ N>\mu  d!C_1^{d-1},
%\,   3\mu  N^3  < \theta  N^{\tau_0},$  so these elements exist provided
%\begin{equation}
%\label{ilrag}\ N>(\mu'+ aN^{4 - \tau } ) d!C_1^{d-1},
%\,   3(\mu'+ aN^{4 - \tau } )  N^3  < (\theta'- aN^{4(1-d\tau)})  N^{\tau_0}
%\end{equation}
%\end{remark}
\begin{remark}\label{dueta}
 Note that if we know that $m,r$  both have an $\ell$ cut with parameters $N,\tet1,\mu,\tau$ then we can deduce that  the subspace $[w_i;q_i]_\ell$ is the one parallel to $[v_i;p_i]_\ell$ and passing through $r$ provided that:
 \begin{equation}
 |r-m| < C_1^{-1}c( N^{4d \tau-1}- C c^{-1} N^{\tau -1})
\end{equation}
notice that   $C_1^{-1}c( N^{4d \tau-1}- C c^{-1} N^{\tau -1})\geq N^\tau$, actually in our computations we will have $|r-m|< N^3$.
\end{remark}

\begin{remark}
With the above lemma we are stating that if $m$  has a $\ell$ cut with parameters $\theta',\mu',\tau$  then, for all choices of $\theta<\theta',\mu'<\mu$, for which $ \tet1,\mu$ are allowable parameters, there exists a spherical neighborhood $B$ of $m$ such that all points $r\in B$  have a $\ell$ cut with parameters $N,\tet1,\mu,\tau$. The radius of $B$ is determined by Formula \eqref{vicin}. Note moreover that if $r$ has a cut $\ell$ for some parameters then so has $-r$ and with the same parameters. Then lemma \ref{mah} holds verbatim if in formula \eqref{vicin} we substitute $|m-r|$ with $|m+r|$.
\end{remark}

The definitions which we have given are sufficient to define and analyze the quasi--T\"oplitz functions, which are introduced in section \ref{toppa}.
In the next subsection we collect some definitions which are useful for the measure estimates and which are independent of the auxiliary parameters $\tet1,\mu$.
\subsection{ Standard cuts}
The following
construction will be useful: we divide  $$[ C N^{4d\tau_0},c
N^{{\tau_1}/4d})=\cup_{i=1}^{d-1}[N^{S_i},N^{S_{i+1}})\cup [N^{S_d},c
N^{{\tau_1}/4d})$$  by setting $N^{S_1}:=C  N^{4d\tau_0}$ and defining
recursively
$$ c^{-1} N^{S_{i+1}}= c^{-1}C \cdot (c^{-1} N^{ S_i})^{4d  }\,,\quad i=1,\dots d-1.$$

 By definition we get $$c^{-1}N^{S_j}=  (c^{-1}C)^{\sum_{i=0}^{j-1}(4d)^i}N^{(4d)^j\tau_0}$$

Recalling that $N>N_0= C c^{-1}$ and   ${\tau_1}= (4d)^{d+1}(\tau_0+1)$, we get
$$c^{-1} N^{S_d}\leq  N^{d(4d)^{d-1}+(4d)^d\tau_0}\leq N^{{\tau_1}/4d}. $$
 We set
$$ \varrho_0:=\tau_0,\; \varrho_d:=\frac{\tau_1}{4d},\quad c N^{ \varrho_i}:=N^{S_i},\ 0<i<d. $$
\begin{Lemma}\label{cutlemma1}For all allowable parameters $c< \tet1,\mu<C$
and for each point $m\frec{N}[v_i;p_i]$  we  construct a  {\em standard cut }
$ \ell , \ 0\leq\ell\leq d$   for $m$    for which the parameter $  \tau $ is one of the previously defined  numbers $ \varrho_i,\ i=0,\ldots,d$.

If $|m|\geq N^{\tau_1}$, then $\ell<d$, if  $ p_1 <  C N^{4d \tau_0}$ then $\ell>0$.
\end{Lemma}
\begin{proof} Let $m\frec{N}[v_i;p_i]$. If $ p_d  \leq c N^{{\tau_1}/4d}$ then we set
$\ell=d$ and $\tau= \varrho_d={\tau_1}/4d$. If $ p_1 \geq C N^{4d \tau_0}$ then we
set $\ell=0$ and $\tau= \varrho_0=\tau_0$.

Otherwise if  $ p_1 < C N^{4d
\tau_0}$ and $ p_d  > c N^{{\tau_1}/4d}$ then  at least one of the $d-1$
intervals $(N^{S_i}, N^{S_{i+1}})$ with $i=1,\dots, d-1$   does not
contain any element of the ordered list $\{ p_2 ,\dots, p_{d-1}\}$.
The parameters  $\ell,\tau$ are fixed by setting $\tau= \varrho_{\bar\imath}$ where $ c N^{ \varrho_{\bar\imath}}=
N^{S_{\bar\imath}}$ and $\bar\imath$ is the smallest among the indices $i$ such
that the interval $(N^{S_i}, N^{S_{i+1}})$ does not contain any
points of the list $\{ p_2 ,\dots, p_{d-1}\}$; finally
$\ell<d$ is the index for which   $ p_\ell\leq N^{S_{\bar\imath}}= c
N^{\tau}$ and  $ p_{\ell+1 }  \geq N^{S_{\bar\imath+1}}=  C (c^{-1} N^{S_{\bar\imath}})^{4d}= C
N^{4d\tau}$.

 If
$ p_d  \leq c{N}^{{\tau_1}\over 4d}$, we apply  Cramer's rule  to the equations $ V m= p$ given by the presentation.  We have $ |m|=
|V^{-1}p|\leq c   d!{N}^{{\tau_1}/4d} (C_1N)^{d-1}< {N}^{\tau_1}$  since $\frac{{\tau_1}}{4d}+d<{\tau_1}$ and as soon as $N>c d!C_1^{d-1}$.\end{proof}

\subsection{Cuts and good points}
As shown in the introduction we need a decomposition of $\Z^d_1$ as in formula \eqref{pustola}. For any given $N$  we set

\begin{equation}\label{nodiv}
 A_0= A_0(N):= \{ m\in \Z_1^d \;:\; m\frec{N}[v_i;p_i]\quad{\rm with} \quad  p_1> C K^{4d \tau_0} \}
\end{equation}
In order to define $A_\ell$ we set
\begin{definition} \label{papilla} For all $[v_i;p_i]_\ell\in \mathcal H_N$ with $1\leq \ell<d$ and  $ p_\ell \leq c N^{{\tau_1}\over 4d}$, the set:
\begin{equation}\label{pippo}
[v_i;p_i]_\ell^{g} := \end{equation}$$ \left\{x\in
[v_i,p_i]_\ell\,\vert \quad  |x|> N^{\tau_1}\,, \; |(v,x)|\geq  C \max(
N^{4d\tau_0},c^{-4d} p_\ell ^{4d}) , \forall v\in B_{N}\setminus
\langle v_i\rangle_\ell\right\} $$ will be called the $N-${\em good}
portion of the subspace $A=[v_i;p_i]_\ell$.\end{definition}

\begin{remark}
 Notice that every $v\in B_{N}\setminus
\langle v_i\rangle_\ell $ gives a non constant linear function $v\cdot x$ on $A$. Thus the good points of $A$ form a non empty open set complement of a finite union of strips  around subspaces of codimension 1 in $A$.
Note moreover that  we are interested only in integral points  and the integral points in $A$ which are not good form a finite union of affine subspaces of codimension one in $A$.
\end{remark}
 \begin{figure}[!ht]
\centering
\begin{minipage}[b]{11cm}
\centering
{\psfrag{A}{$A_0$}
\psfrag{a}{$\mathtt j_j$}
\psfrag{b}{$\mathtt j_i$}
\psfrag{c}{$ a_2$}
\psfrag{d}{$ b_2$}
\psfrag{e}{$ a_1$}
\psfrag{f}{$ b_1$}
\psfrag{H}{$H_{i,j}$}
\psfrag{S}{$S_{i,j}$}
\psfrag{m}{$ \mathtt j_j-\mathtt j_i$}
\psfrag{l}{$ \mathtt j_j+\mathtt j_i$}
\includegraphics[width=11cm]{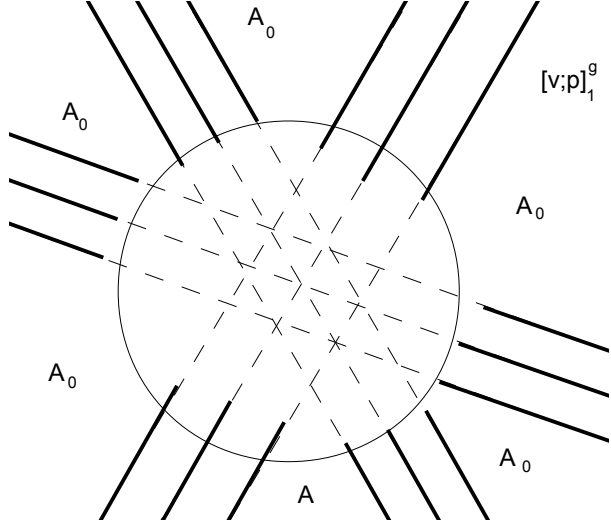}}
\caption{\footnotesize{A drawing of the standard decomposition in $\Z^2_1$. $A_0$ is $\Z^2_1$ minus the dashed lines (each dashed line is described by an equation $[v;p]_1$). On each dashed line the set $[v;p]_1^g$ is signed in solid boldface. Note that $[v;p]_1^g$ is $[v;p]_1\cap \Z^2_1$ minus a finite number of subspaces of codimension two, i.e. points.} \label{fig1}}
\end{minipage}
\end{figure}
\begin{Lemma}\label{minko} Given $p\leq  c N^{\tau_1/(4d)}$ we  fix   $\tau(p)$ so that $N^{\tau(p)}=\max(N^{\tau_0}, c^{-1} p)$ (note that $\tau_0\leq \tau\leq \tau_1/(4d)$).  The following holds:
 for all $c< \tet1,\mu<C$
and for all affine subspaces $[v_i;p_i]_\ell\in {\mathcal H}_N $ such that $  p_\ell =p$, we have that  every point $m\in [v_i;p_i]_\ell^g$  is an $(N, \tet1,\mu,\tau(p))$--good point for $[v_i;p_i]_\ell$.
\end{Lemma}
\begin{proof}
By hypothesis (Formula \eqref{pippo}) $$p_{\ell+1}= (v_{\ell+1},m) \geq  C \max(
N^{4d\tau_0},c^{-4d} p ^{4d}) , $$ recall that $p_\ell=p$. If $  p \leq  cN^{\tau_0}$ then $\tau(p)=\tau_0$ by definition. Since  $p_{\ell+1}\geq  CN^{4d\tau_0}$   $m$ has the cut $\ell$ for all  choices of $c< \tet1,\mu<C$.  Otherwise $c N^{\tau_1/(4d)}  \geq p > cN^{\tau_0}$ and $ p_{\ell+1}\geq  C c^{-4d} p ^{4d}$.
   So in conclusion for all $c< \tet1,\mu<C$
 we have $p_\ell=p= c N^{\tau(p)}< \mu N^{\tau(p)}$  and $p_{\ell+1}\geq  C N^{4d{\tau(p)}}> \theta  N^{4d{\tau(p)}}$, hence the cut.
\end{proof}

We now show that Formula \eqref{pustola} provides a decomposition of $\Z^d_1$.
\begin{Proposition}\label{key2}
Each point $m\frec{N}[v_i,p_i]$ with $|m|>N^{\tau_1}$ and $ p_1 <C N^{4d\tau_0}$ belongs to the set $[v_i;p_i]_\ell^g$  for some choice $0<\ell<d$.
\end{Proposition}
\begin{proof}
According to Lemma \ref{cutlemma1},  each point $m$ has a normalized cut  $0<\ell<d$ for all allowable $ \tet1,\mu$ and for some  $\tau_0< \tau< \tau_1/4d$ with $\tau$ in the finite list $\{\varrho_1,\dots,\varrho_d\}$. Thus for all $w\in \Bn \setminus \langle v_i\rangle _\ell$ we have
$ |(m,w)|> \theta N^{4d\tau}$ for all $\theta<C$, moreover $p_\ell <\mu N^\tau$ for all $\mu>c$. Hence  $ |(m,w)| \geq  C N^{4d\tau}> C N^{4d\tau_0}$ and $ p_\ell \leq c N^\tau$.   Combining these relations we obtain $$|(m,w)|\geq  C\max(N^{4d\tau_0}, c^{-4d} p_\ell ^{4d}),$$ hence $m\in [v_i;p_i]_\ell^g$ by Definition \ref{papilla}.
\end{proof}

\begin{Lemma}\label{lintor}
 Given $p\leq  c N^{\tau_1/4d}$ fix $\tau(p)$ as in Lemma \ref{minko}, then the following holds. Given  $m\in \Z^d_1$  with $m\in [v_i;p_i]_\ell^g$ and $p_{\ell}=p$, then for all $r\in \Z^d_1$ and
 for all parameters $c< \tet1,\mu<C$ such that
\begin{equation}\label{ladisa}
|r-m|<C_1^{-1}(\mu-c) N^{\tau_0-1},C_1^{-1}(C-\theta)N^{4d\tau_0-1} ,
\end{equation}
$r,m$ have the same cut $ \ell$ with parameters $(N, \tet1,\mu,\tau(p)) $ with parallel corresponding affine spaces.\end{Lemma}
\begin{proof}
We can apply   Lemma \ref{minko} to  $m$, obtaining the cut $\ell$ with parameters $(N,\theta',\mu',\tau)$   for all $c<\theta',\mu'<C$. Then, we may apply Lemma \ref{mah} obtaining the required cut for $r$  for any choice of $\tet1,\mu$ satisfying Formula \eqref{vicin} with respect to  $\theta',\mu'$.
 Since  $\theta',\mu'$ can be taken arbitrarily close to $c,C$ Formula \eqref{vicin} follows from Formula \eqref{ladisa}.
 \end{proof}

\section{Quasi--T\"oplitz functions\label{toppa}}

Now and in the following we fix $c=\frac12$, $C=4$.

%In the first definition of quasi--T\"oplitz I do NOT use $A(m)$ but only the affine subspace $[v_i;p_i]$  $A$ apperas  in Corollary \ref{coro}
%{\bf all the notations are different}
%Let $( {K}, \tet1,\mu)$ be parameters such that  $\frac12< \tet1,\mu<4$, $4 {K}^4<  {K}^\td/2$.

\begin{definition}\label{taubilinear}
Given $N, \tet1,\mu,\t1$ such that   $1/2< \tet1,\mu<4$,
$\tau_0\leq \t1\leq \td/4d$ and  $ 4 N^3< \frac 12 N^\td$ we
say that a monomial
$$ e^{\ii (k,\teta)}I^lz^\alpha{\bar z}^\beta z_m^\sigma z_n^{\sigma'}$$ is $(N, \tet1,\mu,\t1)$--bilinear
 if it satisfies momentum conservation \eqref{mome} i.e.
$$ \sigma m+\sigma' n =-\pi(k,\alpha,\beta),$$
\begin{equation}\label{zerouno}
|k|<  N\,,\qquad |n|,|m| >  \theta N^\td \,,\qquad \sum_j |j|
(\alpha_j+\beta_j)< \mu N^3 \,.
\end{equation} and moreover
  there exists $0< \ell<d$ such that both $n,m$ have a $\ell$ cut with parameters $N, \tet1,\mu,\tau$.
By convention  if $m\frec{N}[v_i;p_i]$ and $n\frec{N}[w_i;q_i]$
 with $(p_1,\cdots,p_\ell,v_1,\cdots,v_\ell)\preceq
(q_1,\cdots,q_\ell,w_1,\cdots,w_\ell)$ we  say that the monomial has the cut
$[v_i;p_i]_\ell$.
(this defines univocally an affine subspace associated to the monomial).
Note that by Lemma \ref{cutlemma1} we are sure that $\ell<d$.
In ${\mathcal A}_{r,s}$ we consider the subspace of
$(N, \tet1,\mu,\t1)$--bilinear functions and call
$\Pi_{(N, \tet1,\mu,\t1)}$ the projection onto this subspace.
  \end{definition}
Notice that by Remark \ref{dueta} the cut
$[w_i;q_i]_\ell$ is completely fixed by $[v_i;p_i]_\ell$ and $\sigma
m+\sigma' n$.
%Consider  parameters $( {K}, \tet1,\mu,\t1)$
%with $\frac12< \tet1,\mu<4\,,\quad 4 {K}^4<  {K}^\td/2.$

Having chosen $1/2,4$ as bounds for the parameters $ \tet1,\mu$ we
will call {\em low momentum variables},   denoted by $w^L$ and
spanning the space $\ell^L_\rho$, the $z_j^\sigma$ such that $|j|<4
N^3$. Similarly  we call {\em high momentum variables},  denoted by
$w^H$ and spanning the space $\ell^H_\rho$,  the $z_j^\sigma$ such
that $|j|> N^\td/2$. Notice that  the low and high variables are
separated.
We may write uniquely
\begin{equation}\label{cocca}
\Pi_{(N, \tet1,\mu,\t1)} f= \sum_{\s,\s'=\pm}\sum_{{ |m|,|n| > \theta N^{\tau_1}\atop\exists \ell:\; m,n \,{\rm have \;a }\, \ell\, {\rm cut}}\atop{{\rm with\, parameters}\, N, \tet1,\mu,\t1 }} f_{m,n}^{\s,\s'}(I,\teta,w^L)  z^\s_m z^{\s'}_n
\end{equation}
where
$$
f_{m,n}^{\s,\s'}(I,\teta,w^L)= \sum_{|k|<N\,,\, |\a|+|\b|< \mu N^3\,,\atop -\pi(k,\a,\b)=\s m+\s' n} f_{m,n,k,\a,\b}^{\s,\s'}(I) e^{\ii \langle k,\teta\rangle} z^\a \bar z^\b \,,
$$
finally $f_{m,n,k,\a,\b}^{\s,\s'}(I)$ is an analytic function of $I$ for $|I|<r^2$.
% for $f\in {\mathcal A}_{r,s}$, we set uniquely$$ \Pi_{(K, \tet1,\mu,\t1)} f = ( w^H, F(I,\teta,w^L)  w^H) \,,\qquad  F(I,\teta,w^L) \in {\rm Mat}(\ell^H_\rho,\ell^H_\rho).$$
%In fact since  by definition $\Pi_{(K, \tet1,\mu,\t1)} f$  is quadratic on the high momentum variables, and regular it is descirbed by a bounded linear operator   $F(I,\teta,w^L)$ from $\ell^H_\rho$ in itself, depending (analytically and regularly) on the low momentum variables. We denote by $F_{n,m}^{\sigma,\sigma'}$ with $\sigma,\sigma'=\pm$ its matrix elements.
\vskip20pt

% We consider functions $g^{\sigma,\sigma'}(h,A;I,\teta,w^L)$, where  $h\in \Z_1^d$ with $|h|<\mu K^3$ and  $A\in \mathcal H_K$ is such that $A\frec{K}[w_i;q_i]_j$ with $|q_j|< 4K^{\td/4d}$:
%$$g^{\sigma,\sigma'}(h,[w_i;q_i]_j;I,\teta,w^L):= \!\!\!\!\!\!\!\!\!\!\!\!\sum_{k\in \Z^b\,\vert\,|k|<K \atop {h=-\pi(k,\alpha,\beta)\atop \sum |j|(\alpha_j+\beta_j)<\mu K^3}}\!\!\!\!\!\!\!\!\!\!\!\! g^{\sigma,\sigma'}_{k,\alpha,\beta}(h,[w_i;q_i]_j;I)e^{\ii \langle k,\teta\rangle}z^\alpha \bar z^\beta \in \C,$$
   Given an affine subspace
$A\frec{N}[v_i;p_i]_\ell$, we construct
$(N, \tet1,\mu,\t1,A)$--restricted  {\em T\"oplitz functions} by
setting:
\begin{equation}\label{restrict}
 g(A,I,\teta,z):=\sum_{n,m,\sigma,\sigma',k,\alpha,\beta}^{(N, \tet1,\mu,\t1,A)}
%\sum_{k,\alpha,\beta: |k|<N\atop \sum_{j\in
%\Z_1^d}|j|(\alpha_j+\beta_j)<\mu N^3\,,\; h=
%-\pi(k,\alpha,\beta)}\!\!\!\!\!\!\!\!\!\!\!\!\!\!\!
\!\!\!\!\!\!\!{g}^{\sigma,\sigma'}_{k,\alpha,\beta}(\sigma m+\sigma'
n,A;I)e^{\ii \langle k,\teta\rangle}z^\alpha\bar
z^\beta z_m^\sigma z_n^{\sigma'}\,,
\end{equation}
here the sum $\sum\limits^{(N, \tet1,\mu,\t1,A)}$ means the sum over those $ n,m,\sigma,\sigma',k,\alpha,\beta$ such that
$e^{\ii \langle k,\teta\rangle}z^\alpha\bar
z^\beta z_m^\sigma z_n^{\sigma'}$ is a
$(N, \tet1,\mu,\t1)$--bilinear monomial with
cut given by $A$.
Finally
$g^{\sigma,\sigma'}_{k,\alpha,\beta}(h,B;I)$ is  an
analytic function of $I$, for $|I|<r^2$, which is  well defined for all $\sigma,\sigma'=\pm 1$, $k\in \Z^b$, $h\in Z^d_1$
$\alpha,\beta\in \N^{\Z_1^d}$ and  $B\frec{N}[w_i;q_i]_\ell\in
\mathcal H_N$  such that $|k|<N$,   $h= -\pi(k,\alpha,\beta)$,
$\sum_{j\in \Z_1^d}|j|(\alpha_j+\beta_j)<\mu N^3$ and  $|q_\ell|<
4N^{\td/4d}$.

Notice that the coefficient $g^{\sigma,\sigma'}_{k,\alpha,\beta}(\sigma m+\sigma'
n,A;I)$ depends on $m,n$ only through  $\sigma m+\sigma'
n,A;I$. The sum $\sum_{n,m,\sigma,\sigma',k,\alpha,\beta}^{(N, \tet1,\mu,\t1,A)}$ instead selects those $m,n$ such that  $|m|,|n|> \theta N^\td$, $|\s m+\s' n|<\mu N^3+N  $, $m,n$  have a cut $\ell,\t1$ and the cut of $m$ is $A$ .
\begin{definition}\label{pitop}
A function $g$ is called   {\em piecewise T\"oplitz} if it is of the form:
$$ g= \sum_{A\in  \mathcal H_N \atop A\frec{N}[v_i;p_i]_\ell \,:\, |p_\ell|< \mu N^{\t1}} g(A,I,\teta,z). $$
We denote the space of piecewise T\"oplitz functions as  $\mathbb F( N,\tet1,\mu,\t1)=\mathbb F\subset \mathcal A_{r,s}$
\end{definition}
\begin{remark}\label{pollen}
  Notice that $\mathbb F( N,\tet1,\mu,\t1)$ is a subset of the  $(N, \tet1,\mu,\t1)$ bilinear
functions.
Hence given $g\in \mathbb F( N,\tet1,\mu,\t1)$  we may write it in the form
 \eqref{cocca}
$$
g= \sum_{\s,\s'=\pm}\sum_{{ |m|,|n| > \theta N^{\tau_1}\atop\exists \ell:\; m,n \,{\rm have \;a }\, \ell\, {\rm cut}}\atop{{\rm with\, parameters}\, N, \tet1,\mu,\t1 }} g_{m,n}^{\s,\s'}(I,\teta,w^L)  z^\s_m z^{\s'}_n
$$
and one has that
 \begin{equation}\label{scotto}
  g_{m,n}^{\s,\s'}(I,\teta,w^L)= g^{\s,\s'}(\s m+\s' n,[v_i;p_i]_{\ell}, I,\teta,w^L):=
\end{equation}
 $$\sum_{|k|<N\,,\, |\a|+|\b|< \mu N^3\,,\atop -\pi(k,\a,\b)=\s m+\s' n} g_{k,\a,\b}^{\s,\s'}(\sigma m+\sigma'
n,[v_i;p_i]_{\ell};I)e^{\ii \langle k,\teta\rangle}z^\alpha\bar z^\beta
$$
 if $ |n|,|m|> \theta N^\td $,  $ m\frec{N} [v_i;p_i] $ and there exists $\ell$ such that $m,n$ have a   cut at $\ell$ with parameters $(N,\tet1,\mu,\t1)$.  Otherwise
$g_{m,n}^{\s,\s'}=0$.

Notice  that $g^{\s,\s'}(\s m+\s' n,[v_i;p_i]_{\ell}, I,\teta,w^L)$ depends on $m,n$ only through the subspace $[v_i;p_i]_{\ell}$ and $\s m+\s' n$. In other words the {\em} quadratic form  representation \eqref{cocca}  of a $(N, \tet1,\mu,\t1)$--piecewise   T\"oplitz function has translation invariance in the sense that   $g_{m,n}^{\s,\s'}= g_{m_1,n_1}^{\s,\s'}$ provided that:  $ \s m+\s' n= \s m_1+\s' n_1$, there exists $\ell$ such that $m,n,m_1,n_1$ all have an $\ell,\tau$ cut   and both $m,m_1$ have
 the same associated subspace $[v_i;p_i]_{\ell}$.\end{remark}

Given $f\in \A_{r,s}$  and $\mathcal F\in{\mathbb F}$, we
define     \begin{equation}\label{limi} \bar f:=N^{4
d\t1}\big(\Pi_{(N, \tet1,\mu,\t1)} f- \mathcal F\big).
\end{equation}

Finally set   \begin{equation}\label{unobisbis} \|X_{f }  \|_{r,s}^T
:=\sup_{N\geq  {K}\,,N\in \N, \atop \tau_0\leq\t1\leq
\td/4d}[\inf_{ {\mathcal F}\in \mathbb F}(
\max(\|X_f\|_{r,s},\| X_{\mathcal F}\|_{r,s},\| X_{\bar
f}\|_{r,s} )) ].
 \end{equation}

 \begin{definition} \label{topbis}
We say that $f\in \A_{r,s}$  is quasi- T\"oplitz  of  parameters
$( {K}, \tet1,\mu)$ if $ \|X_f\|_{r,s}^T <\infty. $   We call
$\|X_f\|_{r,s}^T$  the quasi-T\"oplitz norm of $f$.
\end{definition}

\begin{remark}\label{sons}
Notice that our definition includes the T\"oplitz and anti-T\"oplitz
functions by setting, for any $N,\tet1,\mu,\t1$, $\mathcal F = \Pi_{(N, \tet1,\mu,\t1)} f$ and
hence $\bar f=0$.
  In the case of  T\"oplitz functions one trivially has $\|X_f\|_{r,s}^T= \| X_f\|_{r,s}$.
\end{remark}
\begin{remark}\label{osserv}
Intuitively a quasi-T\"oplitz function is a function whose bilinear part is ``well approximated" by a piecewise T\"oplitz function.

Given $K,\theta,\mu$  and a function $f\in \mathcal A_{r,s}$ we proceed as follows. For any choice of $N>K$ and $ \tau_0\leq\tau\leq \tau_1/4d$  we compute a ``weighted distance"  between  $\Pi_{N,\theta,\mu,\tau}f$ and the subspace $\mathbb F$. First, for any  $\mathcal F\in \mathbb F$, we define
$\bar f:=  N^{4d\tau}(\Pi_{N,\theta,\mu,\tau}f-\mathcal F)$ and compute
$\| X_{\bar f}\|_{r,s}$( since $f$ and $\mathcal F$ are in $\mathcal A_{r,s}$  all this quantities are finite); then, in order to obtain a  ``distance",   we perform the infimum over $\mathcal F\in \mathbb F$. Essentially a function $f$ is quasi-T\"oplitz if this weighted distance stays bounded as $N\to \infty$. Note that one could probably prove that the $\inf$  in our definition is actually a $\min$, thus associating  to $f$ a  ``canonical choice'' $\mathcal F$  (depending on $N,\tet1,\mu,\tau$), this however is not needed in our construction, we only need a weaker decomposition as follows.

If $f$ is quasi-T\"oplitz with parameters $(K,\tet1,\mu)$  then for any  $N\geq K$ and $ \tau_0\leq\tau\leq \tau_1/4d$ there exist functions $\mathcal F \in \mathbb F(N,\tet1,\mu,\t1)$, such that setting
$$
\bar f:=  N^{4d \t1}( \Pi_{N,\tet1,\mu,\t1}f- \mathcal F) \,,\quad {\rm we \, have}\quad \|  X_{\mathcal F}\|_{r,s},\|  X_{\bar f}\|_{r,s}< 2  \|X_f\|_{r,s}^T .
$$
\end{remark}
We now concentrate on the very special case of {\em diagonal quadratic functions} $Q(z) :=  \sum\limits_{m\in \Z^d_1}Q_m z_m\bar z_m $. We notice that  in this case we may reformulate the projection on $(N,\tet1,\mu,\tau)$--bilinear functions as:
$$\Pi_{(N,\tet1,\mu,\tau)}Q(z) = \sum_{A\frec{N}[v_i;p_i]_\ell\in \mathcal H(N)\atop |p_\ell|\leq \mu N^{\t1}} \sum_{m\in \Z^d_1}^{(N,\tet1,\mu,\tau,A)} Q_m z_m\bar z_m $$
where $\sum\limits^{(N,\tet1,\mu,\tau,A)}_{m}$  coincides with  $\sum\limits^{(N,\tet1,\mu,\tau,A)}_{m,m,+,-,0,0,0}$ of formula  \eqref{restrict} namely it is the sum over  those $m$ with
$|m|> \theta N^\td$ which  have an  $\ell$ cut with parameters
$( N,\tet1,\mu,\tau)$ associated to the affine space $A$.
\begin{Lemma}\label{diago}
Let $ Q(z)$ be a  quasi-T\"oplitz
 { diagonal quadratic function}.
There exist two  {\em diagonal quadratic functions} ${\mathcal Q}(z)\in \mathbb F$, $\bar Q(z)$:
\begin{equation}\label{pargol}
 {\mathcal Q}(z)= \sum_{A\frec{N}[v_i;p_i]_\ell\in \mathcal H(N)\atop |p_\ell|\leq \mu N^{\t1}}\sum^{(N,\tet1,\mu,\tau,A)}_{m\in \Z^d_1} \mathcal Q(A) z_{m}\bar z_m\,,\end{equation}
 $$  N^{-4d\tau} \bar Q(z) = \Pi_{(N,\tet1,\mu,\tau)}Q(z)- \mathcal Q(z) \,,$$

  such that for all $m$ which have a cut at $\ell$ with parameters $(N,\tet1,\mu,\tau)$   associated to $A$ one has
\begin{equation}\label{appr1}
 Q_m = \mathcal Q(A) + N^{-4d \tau} \bar Q_m .
\end{equation}
 Moreover one has
\begin{equation}\label{appr2}
 |Q_m|,|\mathcal Q(A)|,|\bar Q_m|\leq 2 |X_Q|^T_r
\end{equation}
\end{Lemma}
\begin{proof}
Since  $Q$ is quasi-T\"oplitz  we may approximate it by a function $\mathcal F\in \mathbb F$; moreover since $Q$ is quadratic and diagonal we may choose $\mathcal F$ of the same form.

Hence we can we can fix  quadratic and diagonal functions $\mathcal Q\in \mathbb F$  and $\bar Q= N^{4d\tau} (\Pi_{N,\tet1,\mu,\tau} Q- \mathcal Q)$ so that $\|X_{\mathcal Q}\|_r, \|X_{\bar Q}\|_r\leq 2 \|X_Q\|^T_r$.
To conclude we need to show that a quadratic, diagonal and piecewise T\"oplitz  $\mathcal Q$ is of the form \eqref{pargol}. Indeed  by Formula \eqref{restrict} an $(N,\tet1,\mu,A)$--restricted T\"oplitz function which is  is quadratic  and diagonal is of the form:
 $$ g(A,z)= g(A) \sum_m^{(N,\tet1,\mu,\tau,A)} z_m \bar z_m$$
Our last statement is proved by  noting that
$$
 \|X_{Q}\|_{r}= 2 \sup_{\|z \|_{\rho}<r}\sum_{h\in \Z^d_1}
 |Q_h| \frac{|z_h|}{r} e^{\rho|h|}\geq |Q_j|
$$
by evaluating at  $z^{(j)}_h:=   \delta_{jh} e^{-\rho |j|}r/ 2$. The same holds for $\mathcal Q$ and $\bar Q$.
\end{proof}

\begin{remark}
It is  interesting to compare the set of quasi-T\"oplitz functions with  the T\"oplitz-Lipschitz functions of \cite{EK}. The first observation is that  the set of quasi-T\"oplitz functions is closed with respect to Poisson brackets, while the T\"oplitz-Lipschitz functions are closed only with respect to to Poisson brackets when one of the functions is quadratic. This is due to the fact that the property  of being quasi-T\"oplitz depends on the idea of $(N,\tet1,\mu,\t1)$ bilinear projection, and not on the Hessian of the function.
Indeed one may easily produce functions which are quasi-T\"oplitz but not T\"oplitz-Lipschitz (even in the class of functions which preserve momentum).
%Consider for instance (even in dimension $d=1$)
%$$f(z,\bar z)=\sum_{m\in \Z}c_m |z_m|^4,$$ which is in $\mathcal A_{r,s}$ provided that
%$$\sup_{||z||_\rho< r}\sum_{m\in \Z}|c_m| |z_m|^2 |z_m| e^{\rho |m|}|m|^{d+1}\leq $$
%$$\sup_{||z||_\rho< r}\sup_{m\in \Z}|c_m| |z_m|^2 r\leq
%\sup_{m\in \Z}|c_m| e^{-2\rho|m|}|m|^{-2d-2}r^3 \leq \infty \,, $$
%it is clear that $f$ is quasi-T\"oplitz since its bilinear projection is always zero. The Hessian of $f$ is the diagonal matrix  diag$(c_m |z_m|^2)$ and there is no reason why this matrix should have any T\"oplitz-Lipschitz  limit. Indeed if we set
%$c_m=  e^{2\rho|m|}|m|^{2d+2}$ and choose $z_m= const\, r\, e^{-\rho |m|}|m|^{-d-1-1/4}$ if $m= k^8$ $z_m=0$ otherwise, we have $|z|_\rho<r$ and
% $\lim_{m\to \infty} m c_m |z_m|^2 =\infty$.

A second more subtle  point is weather the class of {\em quadratic}  quasi-T\"oplitz  and T\"oplitz-Lipschitz functions  coincide, this should be true at least for $d\leq 2$ and we expect some inclusions to hold even in higher dimension.

\end{remark}

%\end{Remark}
\section{An abstract KAM theorem}
 The starting point for our KAM Theorem is a family
of Hamiltonians
\begin{equation}\label{hamH} H={\mathcal N}+P,\quad {\mathcal N}
=\la\omega(\xi),I\ra+ \sum_{n\in \Z^d_1}\Omega_n(\xi)z_n\bar z_n,\quad P=
P(I,\teta,z,\bar z, \xi).
\end{equation}
defined in $ D(r,s) \times  \mathcal O$, where $\mathcal O\subset
\R^b$ is open and bounded, say it is contained in a set of diameter $D$.  The functions $\omega(\xi),\Omega_n(\xi)$ are well defined for
$\xi\in \mathcal O$.

It is well known that, for each $\xi\in \mathcal O$, the Hamiltonian equations of motion
for the unperturbed ${\mathcal N}$
%\begin{eqnarray}
%\frac {d\teta}{dt}= \omega,   \frac {dI}{dt}&=0, \frac {d z_n}{dt}=
%-{\rm i}\Omega_n z_n, \frac {d z_n}{dt}= -{\rm i}\Omega_n z_n
%\end{eqnarray}
admit the special solutions  $(\teta, 0, 0, 0)\to (\teta+\omega(\xi) t,
0,0,0)$ that  correspond  to
  invariant tori in the phase space.

\sss Our aim is to prove that,  under suitable hypotheses, there is a set   $\mathcal O_\infty\subset \mathcal O$   of positive Lebesgue measure, so that,   for all
 $\xi \in \mathcal O_\infty$  the Hamiltonians $H$ still admit
 invariant tori.

 \medskip
 We require the following hypotheses on ${\mathcal N}$ and $P$.

\bs \noindent $(A1)$\quad {\it Non--degeneracy:} The map $\xi\to
\omega(\xi)$ is a $C^1_W$ diffeomorphism between $\mathcal O$ and
its image with $|\ome|_{C^1_W},| \nabla \ome^{-1}|_{\mathcal O}\leq
M$.

\bs \noindent $(A2)$\quad  {\it Asymptotics of normal frequency:}

\begin{equation}\label{asymp1} \Omega_n(\xi)=|n|^2+\tilde
\Omega_n(\xi),
\end{equation}
where $\tilde\Omega_n$'s are $C^1_W$  functions of $\xi$ with
 $C^1_W$-norm uniformly bounded  by some  positive constant $L$ with $LM< \1$.

 \bs \noindent$(A3)$\quad  {\it Momentum conservation:} The perturbation $P$
satisfies momentum conservation, it is real analytic  and $C^1_W$ in
$\xi\in \mathcal O$. Namely $P\in \A_{r,s}$.
% with $$\|X_P\|_{\!{}_{D(r,s), \mathcal O}}<\infty.$$

\bs \noindent$(A4)$\quad  {\it Quasi-T\"oplitz property and
Regularity:} the functions $P$ and $\sum_j\tilde\Omega_j|z_j|^2$ are
quasi-T\"oplitz with parameters $( {K}, \tet1,\mu)$ where $$
\1<  \tet1,\mu<4\,,\qquad(\mu-\1) {K}^{\tau_0},
(4- \theta) {K}^{4d\tau_0}>5  {K}^4.$$ One has the bounds:
$$ \|X_{ P}\|^T_{\!{}_{D(r,s), \mathcal O}}<\infty\,, \| \langle\tilde\Omega z,z\rangle\|^T_{\!{}_{D(r,s), \mathcal O}}<L$$

%\bs \noindent $(A5)$ {\it Special form of the perturbation:} The
%perturbation term $ P$ is taken from a special class of analytic
%functions,
% \begin{eqnarray*} \mathcal D=\{P: P=\sum_{k\in\Z^b,l\in
%\N^b, \alpha, \beta}P_{kl\alpha\beta }(\xi)I^le^{{\rm i}
% \la k,\theta\ra}z^\alpha\bar z^\beta\}\end{eqnarray*}
% where $k, \alpha, \beta$ has the following relation
% \beq\label{compactnessP}
%   \sum_{j=1}^b k_ji_j+\sum_{n\in \Z^d_1}(\alpha_n-\beta_n)n=
%  0.
% \eeq

\bs \noindent Now we state our infinite dimensional KAM theorem.
%Assume the Hamiltonian ${\mathcal N}+P$ in (\ref{hamH}) satisfies $(A0-A4)$. If
%the three constants
%  $\gamma,\epsilon,L$ appearing in the previous hypotheses satisfy some smallness condition,
%  which will be determined in the course of the analysis of the KAM algorithm,  then the following holds true:

\begin{Theorem}\label{KAM}
% For $\gamma> 2 \cdot 8^{4d}L$ and $\eps\gamma^{-2}$ small enough,
Assume Hamiltonian ${\mathcal N}+P$ in (\ref{hamH}) satisfies $(A1-A4)$. Let
$\gamma>0$  small enough,  there exists a positive constant
$\varepsilon=\varepsilon(\gamma,b,d,L,M, {K}, \tet1,\mu)$
such that: if $\|X_P\|_{\!{}_{D(r,s), \mathcal O}}^T\leq
\varepsilon$, then there exists a Cantor set $\mathcal
O_\gamma\subset\mathcal O$ with ${\rm
 meas}(\mathcal O\setminus \mathcal O_\gamma)=O(\gamma)$
 and two maps $($analytic in
$\teta$ and $C_W^1$ in $\xi )$
\noindent $$\Psi: \T^b\times \mathcal O_\gamma\to D(r,s),\ \ \ \
 \tilde\omega:\mathcal O_\gamma\to \R^b,$$ where $\Psi$ is
 $\frac{\varepsilon}{\gamma^2}$-close to the trivial embedding
$\Psi_0:\T^b\times \mathcal O\to \T^b\times\{0,0,0\}$ and $\tilde
\omega$ is $\varepsilon$-close to the unperturbed frequency
$\omega$, such that for any $\xi\in \mathcal O_\gamma$ and
$\teta\in \T^b$, the curve $t\to \Psi(\teta+\tilde\omega(\xi)
t,\xi)$ is a linearly stable quasi-periodic solution of the
Hamiltonian system governed by $H={\mathcal N}+P$.
\end{Theorem}

\subsection{Application to the NLS}

The NLS \eqref{equ} is a Hamiltonian equation. We expand the solution in Fourier series as $u=\sum\limits_{m\in \Z^d} u_m\phi_m(x)$
and obtain that the $u_m(t)$ are the Hamiltonian flow of
\begin{equation}\label{hamam}
 N+P= \sum_{i=1}^b (|\mathfrak n_i|^2+\xi_i)|u_{\mathfrak n_i}|^2 +\sum_{n\in \Z^d_1}|n|^2 u_n\bar u_n+\int_{\T^d} g(|\sum\limits_{m\in \Z^d} u_m\phi_m(x)|^2)dx
 \end{equation}
 %\sum_{m\in \Z^d_1} |m|^2 |u_m|^2+ \sum_{m_i\in \Z^d\,,\;m_1-m_2+m_3-m_4=0} u_{m_1}\bar u_{m_2} u_{m_3}\bar u_{m_4}
with respect to the symplectic form $\ii \sum_{m\in \Z^d} d u_m\wedge d \bar u_m$. Here  $g$ is a primitive of the analytic function $f$ so it has a zero of degree at least two.  The conservation of momentum follows by translation invariance.

As an example, if $f(u)=|u|^2u$, then  $P=\sum\limits_{m_i\in \Z^d\atop m_1-m_2+m_3-m_4=0} u_{m_1}\bar u_{m_2} u_{m_3}\bar u_{m_4}$, and the constraint $ m_1-m_2+m_3-m_4=0$ ensures that $P$ satisfies momentum conservation. We introduce standard action-angle coordinates:
$u_{\n_j}=\sqrt{I^{(0)}_j+I_j}e^{i\teta_j},$ $j=1,\cdots,b$; $u_n=z_{n},n\neq \{\n^{(1)},\cdots, \n^{(b)}\}$ where $4 r^2 >I^{(0)}_i> 2 r^2 $ and obtain equations \eqref{equ-h}, where $P$ is the last summand of \eqref{hamam}.  Let us suppose without loss of generality that $g(y)= y^p+O(y^{p+1})$,  so that $P$ is  regular and $X_P$ is of order $|I_0|^{2p} r^{-2}$. It is easily seen that $P$ is   T\"oplitz (hence by Remark \ref{sons} $P$ is quasi-T\"oplitz for all choices of $\tet1,\mu$).   Conditions $(A1)$--$(A4)$  hold with $M=1$ and any $L$ (since $\tilde\Omega=0$).

 In order to apply Theorem \ref{KAM} we fix $r= c \varepsilon^{1\over4p-2}$,  with $c$ small.  We have
$\|X_P\|_{r,s}^T \leq C| I_0|^{2p} r^{-2}$ so the smallness condition is achieved.

\section{KAM step}

%We prove Theorem \ref{KAM} by an iterative procedure. We produce a
%sequence of hamiltonians $H_\nu= {\mathcal N}_\nu+P_\nu$, each well defined in
%some set $D(r_\nu,s_\nu)\times \mathcal O_\nu$ and satisfying
%$(A0)-(A4)$.  We have  $\|P_\nu\|_{r_\nu,s_\nu}\tau_0 0$ and
%inductively $ r_{\nu+1} \leq r_\nu, s_{\nu+1} \leq s_\nu,\ \mathcal
%O_{\nu+1}\subset   \mathcal O_{\nu }$ and  setting
%$r_\infty:\inf(r_\nu),\  s_\infty:\inf(s_\nu),\ \mathcal
%O_{\infty}:=\cap_\nu   \mathcal O_{\nu }$  this sequence is
%convergent in $D(r_\infty,s_\infty)\times \mathcal O_\infty$.  By
%definition one has that $H_{\nu+1}:= e^{\{F_\nu,\}}H_\nu$ where the
%generating function $F_{ \nu}$ is defined through the {\em
%homological equation}.  This is the standard  KAM step, and for
%simplicity of notation,  it is habitual to denote the quantities in
%the $\nu$-th step  without subscript, i.e. $\mathcal O_\nu=\mathcal
%O$, $\ome^\nu=\ome$ and so on.  The quantities in the $(\nu+1)^{th}$
%step are denoted with the subscript $"+"$. We describe in detail the
%KAM step, most of the procedure is completely standard, see
%\cite{GN3} for proofs, the only new part  is to  show that  property
%$(A4)$ for $P,\tilde\Ome$ implies  that $\mathcal O^+$  satisfies
%$(A0)$ and then prove that $(A4)$ holds for $P^+,\tilde\Ome^+$ for
%some new parameters $( {K}^+, \theta^+,\mu^+)$.

Theorem \ref{KAM} is proved by an iterative procedure. We produce a
sequence of hamiltonians $H_\nu={\mathcal N}_\nu+P_\nu$ and a sequence of symplectic
transformations $X_{F_{\nu-1}}^1H_{\nu-1}:= H_{\nu}$, well defined on a domain $D(r_\nu,s_\nu)\times \mathcal O_\nu$. At each step, the perturbation becomes
 smaller  at cost of  reducing the analyticity  and parameter domain. More precisely, the perturbation
should satisfy
$\|X_{P_{\nu+1}}\|_{D(r_{\nu+1},s_{\nu+1}),\mathcal O_\nu}^T\leq
\varepsilon_\nu^{\kappa}, \kappa>1$.
The sequence $r_\nu\to 0$ while $s_\nu\to s/4$ and $O_\nu\to \mathcal O_\infty$.
For simplicity of notation,  we denote the quantities in the
$\nu$-th step  without subscript, i.e. $\mathcal O_\nu=\mathcal O$,
$\ome_\nu=\ome$ and so on.  The quantities in the $(\nu+1)^{th}$
step are denoted with subscript $"+"$. Most of the KAM procedure is
completely standard, see \cite{GY3} for proofs.  The new part is: 1. to
show that Quasi T\"oplitz property $(A4)$ for $P$ and
$\langle\tilde\Ome z,\bar z\rangle$ is kept by KAM iteration and
2. prove the measure estimate using the Quasi T\"oplitz property.

For simplicity, below we always use the same symbol  $C$ to denote
constants independent on the iteration.

\vskip20pt
 {\it One step}\qquad Suppose that the Hamiltonian \eqref{hamH}, well defined in
$D(r,s)\times\mathcal O$, satisfies $(A1-A4)$. Moreover $P$ and $\langle
\tilde \Omega z,\bar z\rangle$ are Quasi T\"oplitz with parameters
$( {K}, \tet1,\mu)$ and we have
\begin{equation}\label{prin}
 |\ome|_{C^1_W},| \nabla \ome^{-1}|_{\mathcal O}\leq M,\quad |\tilde\Omega_n|_{C_W^1}\leq L,
\end{equation}
$$
 \|\langle
\tilde\Omega z,\bar z\rangle\|^T_{D(r,s),\mathcal O}\leq L,\quad
\|X_P\|^T_{D(r,s),\mathcal O}\leq \varepsilon. $$

 Our aim is to
construct:  (1) an open  set $\mathcal O_+\subset \mathcal O$ of
positive measure, (2) a 1-parameter group of symplectic
transformations $\Phi_F^t$, well defined for all $\xi\in\mathcal
O_+,\ t\leq 1$ , such that $ \Phi_F^1 H:=H_+={\mathcal N}_++P_+$ still
satisfies $(A1)-(A4)$ in the domain $D(r_+,s_+)$. Finally $P_+$ and $\langle
\tilde \Omega^+ z,\bar z\rangle$ are Quasi T\"oplitz with new
parameters $( {K}_+, \theta_+,\mu_+)$, and we have
$$|\ome_+|_{C^1_W},| \nabla \ome_+^{-1}|_{\mathcal O}\leq M_+; |\tilde\Omega_n^+|_{C_W^1},\|\langle \tilde\Omega^+ z,\bar
z\rangle\|^T_{D(r_+,s_+),\mathcal O_+}\leq L_+$$
$$\|X_{P_+}\|^T_{D(r_+,s_+),\mathcal O_+}\leq
\varepsilon_+=\varepsilon^\kappa. $$

Let us define $$ R:= \!\!\!\!\!\!\!\!\sum_{k,2|p| +|\alpha|+|\beta|\leq 2} P_{k,p,\alpha,\beta}e^{\ii\langle k,\teta\rangle}I^pz^\alpha\bar z^\beta \,,\quad  \langle R\rangle:= \!\sum_{i=1}^b P_{0,e_i,0,0}I_i+ \sum_{j\in \Z_1^d}P_{0,0,e_j,e_j}|z_j|^2$$
\begin{Remark}
 The quadratic function $R$ is quasi-T\"oplitz and satisfies the bounds $\|X_R\|^T_{r,s}\leq 2 \|X_P\|^T_{r,s}$.
\end{Remark}
The generating function of our symplectic transformation, denoted by $F$, solves the ``homological equation'':
\begin{equation}\label{effehom}\{{\mathcal N},F\}= \Pi_{\leq  {K} } R- \langle R\rangle\end{equation} where
 $\Pi_{\leq  {K}}$  is the projection which collects all terms in $R$ with $|k|\leq  {K}$  and $ {K}$
is fixed to be the quasi-T\"oplitz parameter of $P,\tilde \Ome$.
It's well known (and immediate)  that $F$ is  uniquely defined by
homological equation  for those $\xi$ such that
 $\langle \ome(\xi),k\rangle + \Omega(\xi)\cdot l\neq 0$. In order to have quantitative bounds, we restrict to a set
  $\mathcal O_+$ where (see Lemma \ref{key}):
\begin{equation}\label{cantor0}
|\langle \ome(\xi),k\rangle + \Omega(\xi)\cdot l|\geq \gamma  {K} ^{-2d\td} \,,\qquad |k|\leq {K} ,\; |l|\leq 2,\; (k,l)\neq 0,
\end{equation}  where $k\in \Z^b$, $l\in \Z^{\Z_1^d}$ and $(k,l=\alpha-\beta)$  satisfy momentum conservation \eqref{mome}.
Then $H$ in the new variables is:
$$H_+:= e^{\{F,\cdot\}}H= {\mathcal N}_++P_+ $$ where ${\mathcal N}_+= {\mathcal N}+ \langle R\rangle$ and $P_+= e^{\{F,\cdot\}}H-{\mathcal N}_{+}$.

\subsection{The set $ O_+$}
%\noindent This section is for measure estimate, it's a preparation
%that the homological solution $F$ is  well defined and satisfies
%Quasi-T\"oplitz property.

The set of no-resonant parameter is defined:
%By definition all points $m\in [v_i;p_i]^{g}$ admit a division $|p_i(m)| \leq |p_j| $ for $i\leq j$ and $|p_i(m)|> |p_j|^{4d}$ for $i>j$.
%For each $[v_i;p_i]^{g}$ we choose one point $m=m([v_i;p_i]^g)$, with $k\in \Z^b$ and define $n([w_i;q_i]^g):= m([v_i;p_i]^g)+ \pi(k) $.

\begin{Definition}\label{oppio}  $\mathcal O_+$ is defined to be the open subset of $\mathcal O$ such that:
\begin{itemize}
 \item[i)] For  all $ |k|< {K} $, $h\in \Z$, $(h,k)\neq (0,0)$.
\begin{equation}\label{zero}
|\langle \ome,k\rangle+ h|> 2\gamma {K} ^{-\tau_0}\,. \end{equation}
\item[ii)] For all  $ |k|< {K} $, $l\in\Z^{\Z_1^d}$,   such that $|l|=1$ and $l,k$ satisfy momentum conservation (i.e.  $l=\pm e_m$ with $ -\pi(k)= \pm m$):
\begin{equation}\label{first}
|\langle \ome,k\rangle +\Omega\cdot l |> 2\gamma  {K} ^{-\tau_0}.
\end{equation}  \item[iii)]For all $ |k|< {K}, |l|=2 $, such that $l,k$ satisfy momentum conservation and moreover $l\neq e_m -e_n $ or $l= e_m -e_n $ and $\max(|m|,|n|)\leq  8  {K} ^{\td}$, we set:
\begin{equation}\label{second}
|\langle \ome,k\rangle +\Omega\cdot l |> 2\gamma {K} ^{-2d\td}.
\end{equation}

\item[iv)] For all $ N$ with $  {K}\leq N\leq 2 {K}^{\td/\tau_0}$,  for all affine spaces  $[v_i,p_i]_\ell$
 in $\mathcal H_{N}$ ($1\leq \ell<d$) with $|p_\ell|< cN^{\td/4d}$ we choose a point    $m^g \in [v_i;p_i]^g_\ell$. For each
 such $m^g$ and for all  $k$ such that $|k|\leq  {K}, $ we require:
\begin{eqnarray}\label{third}
|\langle \ome,k\rangle+\Omega_{m^g }-\Omega_{n^g }|&>& 2\gamma \min( N^{-2d\tau_0},2^{-4d}|p_\ell|^{-2d})
%&=:&2\gamma N^{-\t1}\nonumber
\end{eqnarray}%All  this bound  are given according to  measure estimate below.
where $n^g= m^g+\pi(k)$ (see Formula \eqref{patacc} for the definition of $\pi(k)$).
\end{itemize}
\end{Definition}
%\begin{remark}
% Note that in items {\it ii)} and {\it iii)} we could restrict only to those indexes $k,l$ which respect momentum conservation.
%\end{remark}
The set $\mathcal O_+$ is defined in order to ensure Lemma \ref{key} below.
\begin{Lemma}\label{key}
For all $\xi\in\mathcal O_+$, for all $k\in \Z^b$, $|k|\leq  {K}$ and
$l\in \Z^{\Z_1^d}$, $|l|\leq 2$ which satisfy momentum conservation,
we have
\begin{equation}\label{zero1}
|\langle \ome,k\rangle +l\cdot\Omega|\geq \gamma  {K}^{-2 d \tau_1} \,.\end{equation}
\end{Lemma}
Before proving the Lemma we give some relevant notations.

We know that $\tilde\Omega(z):=\sum\limits_m \tilde\Omega_m|z_m|^2 $ is quasi-T\"oplitz quadratic and diagonal,  hence  given $\tet1,\mu,\tau$, we apply Lemma \ref{diago} with $Q(z)=\tilde \Omega(z)$
 to obtain the bounds \eqref{appr1} and \eqref{appr2} for all $m\frec{N}[v_i;p_i]$ which have a  cut  at $\ell$ with parameters $(N,\tet1,\mu,\tau)$:
 \begin{equation}\label{ometg}
  \tilde\Omega_m= {\it  \hat\Omega}([v_i;p_i]_\ell)+  {N}^{-4d\t1}\bar\Omega_m .
\end{equation}
Let us fix an affine subspace $A\frec{N}[v_i;p_i]_\ell$. By Lemma \ref{minko} there exists $\tau:=\tau(p_\ell)$ (depending only on $p_\ell$) such that every $m\in [v_i;p_i]_\ell^g$ has a cut at $\ell$ with parameters $(N,\tet1,\mu,\tau(p_\ell))$ for all $\frac12<\tet1,\mu<4$, hence:
\begin{equation}\label{omet}
  |\tilde\Omega_m- {\it  \hat\Omega}([v_i;p_i]_\ell)|<  2L  {N}^{-4d\t1(p_\ell)},
\end{equation}
  here ${\it  \hat\Omega}([v_i;p_i]_\ell)$ plays the role of $\mathcal Q(A)$ while by \eqref{prin} $L$ dominates the T\"oplitz norm of $\tilde \Omega$.
Note that in particular this relation  holds for $m^g$.

\begin{proof}{(\it Lemma \ref{key})}
The cases with $|l|=0,1$ follow trivially  from the definitions \eqref{zero} and \eqref{first} since $\td $ is  large
with respect to $\tau_0$; same for $\pm l= e_m+e_n$ and $l= e_m-e_n$
with $\max(|m|,|n|)< 8 {K}^{\td}$.

%Let us fix  a $k$ and a subspace $[v_i;p_i]$ and show that \eqref{third} implies that \eqref{zero1} holds for all $l=e_m-e_n$ such that $m\in [v_i;p_i]_j^g$ and  $n=m+\pi(k)$.

For the remaining cases we proceed in two steps: first we fix  $k$,
$N= {K}$ and  one subspace $A\frec{K}[v_i;p_i]_\ell$, we consider \eqref{third}
with this choice of $k,[v_i;p_i]_\ell$. We show that this inequality
implies that \eqref{zero1} holds for all  $l=e_m-e_n$ such that
$m\in [v_i;p_i]_\ell^g$ and $n=m+\pi(k)$.
  We prove this fact by  using  \eqref{omet}  with $N=K$. Finally Proposition \ref{key} ensures that every point $m\notin A_0$ with
$|m|>4 {K}^{\td}$  must belong to some $[v_i;p_i]^g_\ell$.

Let $m$ be any point in  $[v_i;p_i]_\ell^g$. Let us first notice
that
\begin{equation}\label{invar}
\langle\ome,k\rangle+|m|^2-|n|^2=  \langle\ome,k\rangle+|\pi(k)|^2-2\langle \pi(k),m\rangle,
\end{equation}
hence \eqref{zero1} with $l=e_m-e_n$ is surely satisfied  if $|(
\pi(k),m)|\geq  2 {K}^3$ because in that case \eqref{invar} is greater
than $ 2 {K}^3- C_1^2 {K}^2-|\ome| {K} > {K}^3$ provided that $ {K}$ is large
with respect to $C_1$ and $\ome$ .

 If on the other hand  $|( \pi(k),m)|<
2 {K}^3$, then $\pi(k)\in B_ {K}^a$ is  in $\langle v_i\rangle_\ell$,
otherwise we would have $|( \pi(k),m)|> \1  {K}^{4d\tau_0}$ by
definition of $[v_i;p_i]_\ell^g$ and recalling that $ {K}^{4d\tau_0}>
4  {K}^3$ by hypothesis. Thus for all $m\in [v_i;p_i]_\ell^g$ either
\eqref{zero1} is trivially satisfied or $$|m|^2-|n|^2=
|\pi(k)|^2-2\langle \pi(k),m\rangle=|\pi(k)|^2-2\langle\pi(k),m^g\rangle,$$ recall that $m^g$
is one fixed point in $[v_i;p_i]_\ell^g$ on which we have imposed
the non--resonance conditions \eqref{third}.

 We apply \eqref{omet} with $N=K$ to $m, m^g$ and $n= m+\pi(k),n^g= \pi(k)+m^g$. We set
$n\frec{ {K}}[w_i;q_i]$,  since
$(\mu-\1) {K}^{\t1(p_\ell)},(4- \theta) {K}^{4d\t1(p_\ell)}> 5  {K}^4$ we may apply
Lemma \ref{lintor} (with $r=n$) to
conclude that $n$ has an  $\ell$ cut $[w_i;q_i]_\ell$ with
parameters $ \tet1,\mu,\tau$. Note moreover that,  by Lemma \ref{mah} (3)
$[w_i;q_i]_\ell$ is completely fixed by $[v_i;p_i]_\ell$ and $k$. We
have
$$ |\tilde\Omega_n- {\it  \hat\Omega}([w_i;q_i]_\ell)|<2L  {K}^{-4d\t1(p_\ell)},$$ and this relation holds also for $n^g =m^g+\pi(k)$.  This implies that
$$ |\tilde\Omega_m-\tilde\Omega_n- \tilde\Omega_{m^g}+\tilde\Omega_{n^g}|\leq  8 L {K}^{-4d\t1(p_\ell)},$$
where by definition of $\t1$, $ {K}^{\t1(p_\ell)}= \max( {K}^{\tau_0},2
|p_\ell|)$ and hence:
$$ |\langle \ome,k\rangle+\Ome_m-\Ome_n|\geq |\langle \ome,k\rangle+\Ome_{m^g}-\Ome_{n^g}|- 8 L  {K}^{-4d\t1(p_\ell)}\geq$$
\begin{equation}\label{bohh2}
{\gamma\over 2} \min( {K}^{-2d \tau_0},2^{-4d} |p_\ell|^{-2d})\geq
\gamma  {K}^{-\td}.
\end{equation}

Now we may apply Proposition \ref{key2} with $N=K$ to conclude that every point $m$ with $|m|> 8 K^\td $ and $p_1<C{K}^{4d \tau_0}$ belongs to some $[v_i;p_i]_\ell^g$.
 So the measure estimates for the points $m$  which fall in this case are covered by \eqref{second}.

Finally if $m\in A_0$ of Formula \eqref{nodiv}, i.e.   If   we have $p_1>  C{K}^{4d \tau_0}$ then
$$ |\pm \langle \ome,k\rangle+\Omega_m-\Omega_n | > |\pm \langle \ome,k\rangle+|\pi(k)|^2-2(\pi(k),m) +\tilde\Ome_m-\tilde\Ome_n|> {K}^{4d\tau_0}- 2 {K}^2  $$  since $\pi(k)\in B_ {K}$ and hence $|(\pi(k),m)|> p_1$.

 %We show that for $K< {K}^{\frac{\td}{ \tau_0}}$ all the points $|m|>2 {K}^{\frac{\td^2}{ \tau_0}}$ must belong to some $[v_i,p_i]_j^g$.
 We have shown that conditions {\it ii)-iv)}
in $\mathcal O^+$ imply \eqref{zero1}.
\end{proof}
\vskip15pt

\begin{remark}
This lemma essentially saying that by improving \textbf{{\em only
one}} non resonant condition \eqref{third}, we impose \textbf{all}
 the conditions
\eqref{zero1} with $l= e_m-e_n$ such that $m\in [v_i;p_i]_j^g$ and
$n=m+\pi(k)$.
\end{remark}
\begin{remark}\label{quiri}
Notice that up to now we only use  \eqref{third} and \eqref{omet} with $N=
K$. Indeed the other non--resonance conditions are only required in
order to show that the quasi--T\"oplitz property is preserved in
solving the homological equation.
\end{remark}

\begin{Lemma}\label{measure}
 The set $\mathcal O_+$  is open and  has $|\mathcal O\setminus \mathcal O_+|\leq C\gamma
 K^{-\tau_0+b+d/2}$.
%For all $\xi\in\mathcal O^+$, we have:
%\begin{equation}\label{zero1}
%|\langle \ome,k\rangle+ \Omega \cdot l|\geq \gamma\mathcal   K ^{-\td_F}\,,\quad (k,l)\neq 0,|l|\leq 2 \,,\quad \td_F= 2d\td\end{equation}
%for all $k\in\Z^b$  such that $|k|\leq  {K}$ and  for all $l$ which satisfy momentum conservation with $k$.
%What's more, for all $\xi\in\mathcal O_+$ the bounds \eqref{zero1} hold with $\omega,\Omega$ and $ {K}$ replaced  by $\omega_+,\Omega^+$ and $ {K}_+$.
\end{Lemma}
For the measure estimates, given $\varrho>0$ we define
$$\Ra^{\varrho}_{k,l}:= \left\{\xi\in \mathcal O \vert\; |\langle \ome,k\rangle + \Omega \cdot l|< \gamma   {K}^{-\varrho}\right\}, $$
\begin{Lemma}\label{misu}
For all $(k,l)\neq (0,0)$ $|k|\leq  {K}$ and $|l|\leq 2$, which satisfy momentum conservation,   one has
$|\Ra^{\varrho}_{k,l}|\leq C \gamma  {K}^{-\varrho}$.
% with $C_2= 2 M^{-1} {(diam \mathcal O)}^{b-1}$.
\end{Lemma}

{\em Proof.}
By assumption $\mathcal O$ is
contained in some open set of diameter $D$.

Choose  $a$ to be a vector such that $\langle k,a\rangle=|k|$, we
have
$$|\partial_{t}(\langle k,\omega(\xi+ta) \rangle + \Omega \cdot l )|
\geq M( |k| -M L)\geq \frac M2.$$ which leads to  $$
\int_{\Ra^{\varrho}_{k,l}}d\xi\leq 2 M^{-1} \gamma    {K}^{-\varrho}
\int_{\xi+ta \cap\Ra^{\varrho}_{k,l}}  dt\int d\xi_2\dots d\xi_b \leq
2 M^{-1} D^{b-1}\gamma {K}^{-\varrho}\eqno\endproof$$ 
%\vskip15pt {\it Proof of Lemma \ref{measure}}:

\begin{proof}{\it  Lemma \ref{measure}.}The first statement is trivial, indeed {\it ii)-iv)} are a
finite number of inequalities; notice that in {\it iv)} for each
$[v_i,p_i]_\ell^g$ and $k$ we  impose only  one condition by
choosing one couple $m^g,n^g$. Finally by Remark \ref{numero} there
are a finite number of $[v_i,p_i]_\ell^g$. Item {\it i)} apparently has infinitely many conditions since $h\in \Z$, however we note that all but a finite number  (i.e. $|h|< 2|\ome| K$) are trivially satisfied.

Let us prove the measure estimates; to impose  \eqref{zero} with $h=0$ we have to remove %\marginpar{with this notation $ {K}^{-\td}|k|^{-1}$!}
\begin{equation} |\cup_{|k|\leq   {K}}\Ra^{\tau_0}_{ k,0}|\leq C(b)\gamma   {K}^{-\tau_0+b}.\end{equation} For  $h\in\Z$  we set
$$\tilde\Ra^{\varrho}_{k,h}:= \left\{\xi\in \mathcal O \vert\; |\langle \ome,k\rangle + h|< \gamma   {K}^{-\varrho}\right\}, $$ and note that $\tilde\Ra^{\varrho}_{k,h}$ is empty if $|h|> 2|\ome||k|$.
 As in Lemma \ref{misu} for fixed $(k,h)$ we have $|\tilde\Ra^{\varrho}_{k,h}|\leq C \gamma  {K}^{-\varrho}$. Then
\begin{equation} |\cup_{|k|\leq   {K}, |h|\leq 2 |\ome| |k| }\tilde\Ra^{\tau_0}_{ k,h}|\leq C(b)\gamma   {K}^{-\tau_0+b+1}.\end{equation}

In order to impose  the first Melnikov condition \eqref{first} we note that  by momentum conservation in $\Ra^{\tau_0}_{k,l}$ we have $l=\pm e_{\mp \pi(k)}$. Then we have to  remove:
\begin{equation}|\cup_{|k|\leq  {K}\,,\; l = \pm e_{\mp \pi(k)}}\Ra^{\tau_0}_{k,l}| \leq C(b)\gamma   {K}^{-\tau_0+b}.\end{equation}
If $l= \pm(e_m+e_n)$ the momentum conservation fixes $ n= \mp \pi(k)-m$;  we notice that the condition
$$|\pm \langle \ome,k\rangle+ |m|^2+|n|^2+\tilde\Omega_m+\tilde\Omega_n|<\frac12 $$ implies $|\pm \langle \ome,k\rangle+ |m|^2+|n|^2|<1$ and hence
$|m|^2+|n|^2<2|\omega|  {K}$, and we have to remove a set of parameters:
\begin{equation}\quad\quad|\cup_{k\leq  {K},\atop l=\pm (e_m+e_n)}\Ra^{\tau_0}_{k,l}|=|\cup_{k\leq  {K}}\cup_{l=\pm (e_m+e_n)\atop |m|\leq  C(b) \sqrt{ {K}}\,, n=-\pi(k)-m }\Ra^{\tau_0}_{k,l}|\leq C\gamma   {K}^{-\tau_0+b +d/2}, \end{equation}
 In conclusion  one gets \eqref{zero} and \eqref{first}
 with $\tau_0>b +d/2$  and $l\neq \pm(e_m -e_n)$ by removing an open set of measure  $C\gamma   {K}^{-\tau_0+b+ d/2}$ .

One trivially has
\begin{equation}|\cup_{k\leq  {K}}\cup_{{l=\pm (e_m-e_n)\,, m- n=\mp \pi(k)\,,}\atop \max( |m|,|n|)\leq  8  {K}^{\td}}\Ra^{2d\td}_{k,l}|\leq C\gamma   {K}^{-d\td+b },\end{equation}
 so we have
\eqref{second} by removing an open set of measure  $C\gamma
 {K}^{-d\td+b }$.

In order to deal with the last case, for all natural $N$ such that $
K\leq N\leq 2 {K}^{\td/\tau_0}$, for all affine subspaces
$[v_i;p_i]_\ell$ and for all $|k|\leq  {K}$ we set
\begin{eqnarray}\label{resonant set}\Ra_{k,[v_i;p_i]^g_\ell}^N:= \{\xi\;\vert\; |\langle \ome,k\rangle+\Omega_{m^g}-\Omega_{n^g}| < 2 \gamma
\min(N^{-2d\tau_0},2^{-4d}|p_\ell|^{-2d})\}\end{eqnarray}

Following Lemma \ref{misu}, $|\Ra_{k,[v_i;p_i]^g_\ell}^N|<C\gamma \min(N^{-2d\tau_0},2^{-4d}|p_\ell|^{-2d})$. By Remark
\ref{numero}
%
%
% $\max(|m|,|n|)\geq  K^\td$, we use quasi-T\"oplitz property.
%For $\forall |k|\leq K$,  if $m\in [v_i;p_i]^g$, with momentum conservation, we have $n\in [w_i;q_i]^g$, which is given by $k$ and $[v_i;p_i]^g$.
%With Lemma \ref{mah}, they have separation condition with same index $j$.
%We define correspondingly $$\Ra_{k,[v_i;p_i]^g}=\{|(k,\omega^+)+\Omega_m([v_i;p_i])^+-\Omega_n(W;q)^+|<\gamma \min(K^{-2d\tau_0},|p_j|^{-2d})$$
 we have:
$$|\cup_{ {K}\leq N\leq  {K}^{\td/\tau_0}}\cup_{\ell=0,\cdots,d-1}\cup_{\1 N^{\tau_0}\leq |p_\ell|\leq 4 N^{\td\over 4d}}\cup_{[v_i;p_i]^{g}_\ell\atop |k|<  {K}} \Ra_{k,[v_i;p_i]_\ell^g}^N|$$
$$ \leq C\gamma \sum_{N\geq  {K}}\sum_{\ell=0}^{d-1}\sum_{|p_\ell|>\1 N^{\tau_0}} |p_\ell|^{-2d-1+d}N^{\ell d} {K}^{b}\leq 4^d C_2\gamma  {K}^{-d\tau_0+b}, $$
so that we have \eqref{third} by removing an open set of measure
$C\gamma  {K}^{-d\tau_0+b}$.
\end{proof}
\vskip15pt

\subsection{Quasi-T\"oplitz property}
The main proposition of our paper is following:\begin{Proposition}\label{pro:1} The functions $P_+$, $\tilde\Ome^+|z|^2$  are quasi-T\"oplitz with parameters $( {K}_+, \theta_+,\mu_+)$ such that:
$$  4 {K}_+< \sqrt{(\mu-\mu_+) }( {K}_+)^{3/2}, \quad 4\mu_+ {K}_+^{4}< ( \theta_+- \theta) {K}_+^{4d\tau_0-1}.$$
\end{Proposition}

The key of our strategy is based on the following three propositions
which are proved in the appendix.

%\begin{Lemma}\label{key2}
% For all $\xi\in\mathcal O_+$, $ K_1\geq  {K}$ and $K_1\in
%\mathbb N$. If  $[v_i;p_i]_\ell$, $\ell< d$, is a  affine space
%associate to $K_1$, then for any $m\in [v_i;p_i]_\ell$, $l= e_m -
%e_n$ satisfy momentum conservation, we have
%\begin{equation}\label{zero2}
%|(k,\omega(\xi))+l\cdot\Omega|\geq 2\gamma \min(
%K_2^{-2d\tau_0},2^{-4d}|p'_{\ell'}|^{-2d}).\end{equation}
% and
%$$|\tilde\Omega_m- {\it  \tilde\Omega}([v_i;p_i]_\ell)|,|\tilde\Omega_n- {\it
%\tilde\Omega}([w_i;q_i]_\ell)|<  L K_1^{-4d\t1},$$
%  $p'_{\ell'}$ is given by certain $[v_i',p_i']_{\ell'}$ associate to
% $K_2$ with $ {K}\leq K_2\leq  {K}^{\td\over\tau_0}$.
%\end{Lemma}

%\begin{remark} This Lemma is a preparation that $F$ Quasi T\"oplitz.
% It says we have separation for all $K\geq {K}$ from finite condition with $ {K}\leq K\leq  {K}^{\td\over \tau_0}$.
%\end{remark}
\begin{Proposition}\label{smaldiv}
For any  $N\geq  {K}$, $k\in \Z^b$ with $|k|< {K}$ and for all
$|m|,|n|\geq  \theta N^\td$ such that $m-n=-\pi(k)$, $m\frec{N}
[v_i;p_i]$, $n\frec{N} [w_i;q_i]$ and $m,n$ have a $\ell$
cut with parameters $ \tet1,\mu,\t1$ for some choice of $\ell,\t1$
one has
$$ |\langle \ome,k\rangle+|m|^2-|n|^2+ \it\hat\Omega([v_i;p_i]_\ell)- \it\hat\Omega([w_i;q_i]_\ell)|=\qquad\quad\qquad\qquad\qquad\qquad\qquad\qquad\qquad\qquad$$ $$
|\langle \ome,k\rangle+|\pi(k)|^2-2\langle \pi(k),m\rangle
+ \it\hat\Omega([v_i;p_i]_\ell)- \it\hat\Omega([w_i;q_i]_\ell)|\geq\qquad\qquad\qquad$$
$$\left\{\begin{array} {l} \gamma {K}^{-2d\td\t1/\tau_0} \,,\;
\pi(k)\in \langle v_i\rangle_\ell\\ \1N^{4d\t1}\,,\quad {\rm
otherwise} \end{array}\right.,\qquad\qquad\qquad\qquad\qquad\qquad$$
where ${\it \hat\Ome}([v_i;p_i]_\ell)$ and
${\it\hat\Ome}([w_i;q_i]_\ell)$ are defined by Formula
\eqref{ometg}.
\end{Proposition}

\begin{Proposition}\label{submain}
For $\xi\in \mathcal O_+$, the solution of the homological equation
$F$ is quasi-T\"oplitz for  parameters $( {K}, \tet1,\mu)$,
moreover one has the bound
\begin{equation}\label{pullo}\|X_F\|^T_{r,s} \leq C\gamma^{-2} {K}^{3\td^2/\tau_0} \|X_P\|_{r,s}^T\,,\end{equation} where $C$ is some constant.
\end{Proposition}

Analytic quasi-T\"oplitz functions are closed under Poisson bracket.
More precisely:
\begin{Proposition}\label{main}
Given  $f^{(1)},f^{(2)}\in \A_{ r,s}$, quasi-T\"oplitz with parameters $( {K}, \tet1,\mu)$    we have that   $\{f^{(1)},f^{(2)}\}\in \A_{r',s'}$,  is quasi-T\"oplitz for all  parameters $( {K}', \theta',\mu')$ such that  ${  {K}'}, \theta',\mu',r',s'$ satisfy:
\begin{equation}\label{pois1}
\frac{1}{( {K}')^2}\leq {(\mu-\mu') }, \quad \frac{2\mu' }{( {K}')^{4d \tau_0-4}}< ( \theta'- \theta)\,,\quad  e^{-(s-s') {K}'}( {K}')^\td<1
%$$ {  {K}'} < (\mu-\mu'){ {K}'}^{3} \,,\;  2\mu' {  {K}'}^{3}< ( \theta'- \theta) {  {K}'}^{ 4d\tau_0-1},$$
%\begin{equation}\label{pois1}
%  e^{-(s-s'){  {K}'}} { {K}'}^\td<1.
\end{equation} We have the bounds
%Moreover, setting $ \Pi_{ {K}'}, \theta',\mu'}\{f^{(1)},f^{(2)}\}:= (w^H, F^{(1,2)} w^H)$ and calling $$C^{(i)}=\max(\|X_{f^{(i)}}\|_{r,s},\| X_{(w^H,{\mathcal F^{(i)}}w^H)} \|_{r,s}\,, \|X_{(w^H,{\bar F^{(i)}}w^H)} \|_{r,s}$$ we have the bounds
%\begin{equation}\label{poisbound}
% \| X_{(w^H,{\mathcal F^{(1,2)}} w^H)}\|_{r',s'},\| X_{(w^H,\bar F^{(1,2)} w^H)}\|_{r',s'} \leq \frac{ 10 C^{(1)}C^{(2)}}{\delta}
%\end{equation}  where $\delta= \min(s-s',\eta-\eta')$
\begin{equation}\label{poisbound}
 \| X_{\{f^{(1)},f^{(2)}\}}\|^T_{r',s'} \leq C_1\del^{-1} \|X_{f^{(1)}}\|^T_{r,s}\|X_{f^{(2)}}\|^T_{r,s}
\end{equation}  where $\delta= (\frac{r'}r)^2 \min( s-s',1-\frac rr')$
\vskip5pt
%\frac{10 C^{(1)}C^{(2)}}{\min(s-s',s_1-s_1')\min(r-r',r_1-r_1')^2}
%(ii) Given $F^{(i)}$, $ i=1,\dots,N$ as in item (i)  the nested Poisson bracket

%$$\{F^{(1)},\{F^{(2)},\dots,\{F^{(N-1)},F^{(N)}\}\dots\}$$ is $K_1$--quasiT\"oplitz/antiT\"oplitz in $\D(r',s')$ if
%$$ %3N\td\ln(K)< \min (K^{c-c'},
%  f' K^{c'} + 2N\td\ln(K) < f K^c \,,\qquad   N f' K_1^{c'}+2N\td \ln(K)< \min\left(( \theta'- \theta) K^{\td}, \,K^{\td-1}/2\right) . $$ for all $K>K_1$
%\vskip5pt
(ii) Given $f^{(1)},f^{(2)}$  as in item (i), with  $C_1 e\|X_{f^{(1)}}\|_{r,s}^T\del^{-1}\ll1 $, the function  $f^{(2)}\circ \phi_{f^{(1)}}^t:= e^{t \{f^{(1)},\cdot\}}f^{(2)}$, for $t\leq 1$, is  quasi-T\"oplitz in $\D(r',s')$ for all parameters $({  {K}'}, \theta',\mu')$ such that
\begin{equation}\label{boh}\quad\frac{(\ln {K} ')^2}{( {K}')^2}\leq {(\mu-\mu') }, \quad \frac{2\mu' (\ln {K} ')^2}{( {K}')^{4d \tau_0-4}}< ( \theta'- \theta)\,,\quad  e^{-(s-s')\frac{ {K}'}{(\ln {K}')^2}}( {K}')^\td<1,
\end{equation} we have the bounds:
$$ \|X_{f^{(2)}\circ \phi_{f^{(1)}}^t}\|^T_{r',s'} \leq (1- C_1 e\del^{-1}\|X_{f^{(1)}}\|_{r,s}^T)^{-1}\|X_{f^{(2)}}\|_{r,s}^T$$
\end{Proposition}

%\end{proof}\qed

\section{Estimate  and KAM Iteration}

\subsection{Estimate on the coordinate transformation}\label{4.2}

 \sss We estimate  $X_F$ and $\phi_F^1$  where $F$ is given by \eqref{effehom}.
\begin{Lemma}\label{Lem4.3}
Let $D_i=D(\frac i4r, s_++\frac{i}4 (s-s_+) )$, $0 <i \le 4$. Then
\begin{equation}\label{4.20}
\|X_F\|_{D_3\times \mathcal O_+}\le c\gamma ^{-2} K^{4d\td}\varepsilon\,,\quad  \|X_F\|_{D_3\times \mathcal O_+}^T\le
C\gamma ^{-2} {K}^{3\td^2/\tau_0}\varepsilon
\end{equation}
\end{Lemma}

\begin{Lemma}\label{Lem4.4}
Let  $\eta=\varepsilon^{\frac 13}, D_{i\eta}= D(\frac i4 \eta
r,s_++\frac {i}4(s-s_+)), 0 <i \le 4$. If $\varepsilon\ll (\frac
12\gamma^2  {K}^{-3\td^2/\tau_0})^{3}$, we then have that
\begin{equation}
\phi_F^t:  D_{2\eta} \to  D_{3\eta} ,\ \ \ -1 \le t\le 1,
\label{4.26}
\end{equation} is an analytic map,
 moreover,
\begin{equation}
\|\phi_F^t(z)-(z)\|_{D_{1\eta}\times \mathcal O_+}\le C\gamma
^{-2} {K}^{4d\td}\varepsilon^{1/3}\,,\quad
\label{4.27}
\end{equation}
\end{Lemma}
\begin{proof}
We first notice that
$$ \| X_F\|_{3\eta}^T \leq c' \eta^{-2} \|X_F\|_{D_3\times \mathcal O_+}^T \leq C \varepsilon^{-2/3}\gamma ^{-2} {K}^{3\td^2/\tau_0} \varepsilon <1$$
by our smallness assumption.
Let us denote by ${\mathcal B}_{2\eta}$ the space of close to identity analytic symplectic maps $D_{2\eta} \to  \C^{2b}\times \ell_{\rho}$ with finite norm \eqref{norvec}. Similarly we call ${\cal C}\big( [0,1], {\mathcal B}_{2\eta} \big)$
 the Banach space
of all continuous functions $t\mapsto \phi^t$ from $[0,1]$ to
${\mathcal B}_{2\eta}$
  endowed with the norm $\sup_{t\in [0,1]}\|\cdot\|_{2\eta}.$ Consider the ball of radius $\rho:=2\|X_F\|_{3\eta}<1$ and centered in $\phi^0= id$. For $\phi^t$ in such ball consider the map
   \begin{equation}\label{cirdan}
P(\phi^t):=id+\int_0^t X_F\circ \phi^s ds
\end{equation}
It is simple to see that the above map is  a contraction, in particular
$$
\sup_{t\in[0,1]}\Big\| \int_0^t X_F\circ \phi^s ds \Big\|_{2\eta}
\leq  \sup_{t\in[0,1]}\|X_F\circ \phi^t \|_{2\eta}
\leq (1+ \rho) \|X_F\|_{3\eta}
\leq \rho\,,
$$
The Lemma follows since the  Hamiltonian flow $\phi_F^t$ generated by $F$ at time $t\in[0,1]$
is found as the fixed point of $P$.
\end{proof}

\subsection{Estimate of the new perturbation}

The symplectic map $\phi_F^1$ defined above transforms $H$
into $H_+={\mathcal N}_+ + P_+$,  where ${\mathcal N}_+={\mathcal N}+\langle R \rangle$ and
 \begin{eqnarray} P_+&=&\int_0^1 (1-t)\{\{{\mathcal N},F\},F\}\circ \phi_F^{t}dt+\int_0^1 \{\Pi_{\leq  {K}}R,F\}\circ \phi_F^{t}dt +(P-\Pi_{\leq  {K}}R)\circ
\phi^1_F\nonumber \\
&=&\int_0^1 \{R(t),F\}\circ \phi_F^{t}dt +(P-\Pi_{\leq {K}}R)\circ \phi^1_F,\label{newpert}
\end{eqnarray}
with $R(t)=(1-t)({\mathcal N}_+-{\mathcal N})+t\Pi_{\leq  {K}}R$. Hence
$$
X_{P_+}=\int_0^1 (\phi_F^{t})^*X_{\{R(t),F\}} dt
+(\phi^1_F)^*X_{(P-\Pi_{\leq {K}}R)}.
$$
\begin{Lemma}\label{estnew}
The new perturbation $P_+$ satisfies the estimate
$$ \|X_{P_+}\|_{D(r_+,s_+)}\le  C\gamma^{-2} {K}^{4d\td}\eps^{4/3}.
$$
\end{Lemma}
{\em Proof}
 According to Lemma \ref{Lem4.4}, $$\|D\phi_F^t-Id\|_{D_{1\eta}}\le
 c\gamma^{-2}
 {K}^{4d\td}\varepsilon^{1/3}, \quad -1\le t\le 1,
$$
thus
$$\|D\phi_F^t\|_{D_{1\eta}}\le 1+\|D\phi_F^t-Id\|_{D_{1\eta}}\le 2, \quad -1\le t\le
1.$$
$$ \|X_{\{R(t),F\}}\|_{D_{2\eta}}\le  \eta^{-2} \|X_{\{R(t),F\}}\|_{D_{2}} \leq  C\gamma^{-2}
 {K}^{4d\td}\eta^{-2}\varepsilon^2,$$ and
$$ \|X_{(P-\Pi_{\leq {K}}R)}\|_{D_{2\eta}}\le C\eta \varepsilon, $$ we
have  $$ \|X_{P_+}\|_{D(r_+,s_+)}\le C\eta \varepsilon +
C(\gamma^{-2} {K}^{4d\td})\eta^{-2}\varepsilon^2 \le
C\gamma^{-2} {K}^{4d\td}\eps^{4/3}.
\eqno\endproof$$

We need to show that $P_+$ is quasi--T\"oplitz and estimate its
T\"oplitz norm.
We notice that $R(t)$  and $P-\Pi_{\leq  {K}} R$ in \eqref{newpert} are quasi--T\"oplitz, by hypothesis $(A4)$. Then, by Proposition \ref{main} {\it ii)}, we have that $R(t)\circ \phi^t_F= e^{\{F,\cdot\}} R(t)$ and   $(P-\Pi_{\leq  {K}} R)\circ\phi^t_F$ are quasi--T\"oplitz as well.
Recalling Proposition \ref{main}, and repeating the reasoning  of Lemma
\ref{estnew} with Quasi- T\"oplitz norm, one has
\begin{Lemma}
 Set $\eps_+:=C\gamma^{-2} {K}^{3\td^2/\tau_0}\eps^{4/3}$,
 then  $$ \|X_{P_+}\|^T_{D(r_+,s_+)}\ \le \eps_+.
$$
%\|X_{P_+}\|_{D(r_+,s_+)}\ \le \eps_+\,,\qquad
\end{Lemma}

\subsection{Iteration lemma}
In order to make the KAM machine work fluently, for any given
$s,\varepsilon,r, \gamma$ and for all $\nu\ge 1$, we define the
following sequences \begin{eqnarray}&&\label{series}
s_\nu=s(1-\sum_{i=2}^{\nu+1}2^{-i}),\nonumber\\
&& r_\nu=\frac
14\eta_{\nu-1}r_{\nu-1}=2^{-2\nu}(\prod_{i=0}^{\nu-1}\varepsilon_i)^{\frac
13}r_0,
\\
&&\varepsilon_\nu=c\gamma ^{-2}
K^{3\td^2/\tau_0}_{\nu-1}\varepsilon_{\nu-1}^{\frac {4}{3}},\quad
\eta_\nu=\varepsilon_\nu^{\frac
13}\nonumber\\
&&M_\nu=M_{\nu-1}+\varepsilon_{\nu-1}, \quad
L_\nu=L_{\nu-1}+\varepsilon_{\nu-1},\nonumber\\
&& \mu_\nu= \mu-\sum_{i=1}^\nu (\chi)^{-i}\,,\quad  \theta_\nu=
 \theta +\sum_{i=1}^\nu (\chi)^{-i}\nonumber\\
&& {K}_\nu=c(s_{\nu-1}-s_{\nu})^{-1}\ln\varepsilon_{\nu}^{-1},
\nonumber \end{eqnarray} where $c,1<\chi<\frac 4 3 $ is a constant,
and the parameters $r_0,\varepsilon_0,L_0,s_0$ and $ {K}_0$
are defined
 to be $r,\varepsilon,L,s$ and bounded by $\ln \varepsilon^{-1}$ respectively.

 We iterate the KAM step, and proceed by induction.

\begin{Lemma}\label{Lem5.1}
 Suppose at the $\nu$--step of
KAM iteration, the hamiltonian
$$ H_\nu={\mathcal N}_\nu+P_\nu,$$
is well defined in $D(r_\nu, s_\nu)\times \mathcal O_{\nu} $, where
${\mathcal N}_\nu$ is usual "integrable normal form", $P_{\nu}$ and
$\sum\tilde\Omega^\nu_n |z_n|^2$  satisfy $(A4)$ for $(
K_\nu, \theta_\nu,\mu_\nu)$,  $\omega_\nu$ and $\Omega_n^\nu$ are
$C^1_W$ smooth
$$|\ome_\nu|_{C^1_W},| \nabla \ome_\nu^{-1}|_{\mathcal O}\leq M_\nu,|\tilde\Omega_n^\nu|_{C_W^1}\leq L_\nu,  \quad
|\Omega_n^\nu-\Omega_n^{\nu-1}|_{\mathcal O_{\nu}}\le
\varepsilon_{\nu-1};$$
$$
\|X_{P_\nu}\|^T_{D(r_\nu,s_\nu),\mathcal O_\nu}\leq \varepsilon_\nu.
\quad\|\langle \tilde\Omega^\nu z,\bar
z\rangle\|^T_{D(r_\nu,s_\nu),\mathcal O_\nu}\leq L_\nu$$

\noindent Then there exists a symplectic and Quasi-T\"oplitz change
of variables for parameter $(
K_{\nu+1}, \theta_\nu,\mu_\nu)$, \beq \Phi_\nu:D(r_{\nu+1},s_{\nu
+1}) \times\mathcal O_{\nu+1}\to D(r_{\nu},s_{\nu}),\eeq where
$|\mathcal O_{\nu+1}\backslash\mathcal O_\nu|\lessdot \gamma
 {K}_{\nu+1}^{-\tau_0+b+\frac d 2}$, such that on $D(r_{\nu+1},s_{\nu
+1})\times\mathcal O_{\nu+1}$
 we have
$$ H_{\nu+1}=H_\nu\circ \Phi_\nu=e_{\nu+1}+{\mathcal N}_{\nu+1}+P_{\nu+1}=e_{\nu+1}+\langle \omega_{\nu+1},I\rangle+\langle\Omega^{\nu+1} z,\bar z\rangle+P_{\nu+1},$$
 with
$\omega_{\nu+1}=\omega_\nu+\sum\limits_{|l|=1}
lP_{0,l,0,0},\; \Omega_n^{\nu+1}=\Omega_n^{\nu}+P_{0,0,e_n,e_n}^\nu$.

${\mathcal N}_{\nu+1}$ is  an "integrable normal form". $P_{\nu+1}$ and
$\sum\tilde\Omega^{\nu+1}_n |z_n|^2$ satisfy $(A4)$ for  parameters
$( {K}_{\nu+1}, \theta_{\nu+1},\mu_{\nu+1})$. Functions
$\omega_{\nu+1}$ and $\Omega_{n}^{\nu+1}$
 are $C_W^1$ smooth
$$|\ome_{\nu+1}|_{C^1_W},| \nabla \ome_{\nu+1}^{-1}|_{\mathcal O}\leq M_{\nu+1}
,|\tilde\Omega_n^{\nu+1}|_{C_W^1}\leq L_{\nu+1},  \quad
|\Omega_n^{\nu+1}-\Omega_n^{\nu}|_{\mathcal O_{\nu+1}}\le
\varepsilon_{\nu};$$
$$
\|X_{P_{\nu+1}}\|^T_{D(r_{\nu+1},s_{\nu+1}),\mathcal O_{\nu+1}}\leq
\varepsilon_{\nu+1},\quad \|\langle \tilde\Omega^{\nu+1} z,\bar
z\rangle\|^T_{D(r_{\nu+1},s_{\nu+1}),\mathcal O_{\nu+1}}\leq
L_{\nu+1}$$
\end{Lemma}

\begin{itemize}
 \item[$\bullet$]
  By Proposition \ref{pro:1}, the new perturbation
$P_{\nu +1}$ and $\langle\tilde\Omega^{\nu+1} z,z\rangle$  satisfy the
Quasi-T\"oplitz property for parameters $(
K_{\nu+1}, \theta_{\nu+1},\mu_{\nu+1})$.
 As we can see, when we require $\td>\tau_0>12$:  $$\forall N\geq  {K}_{\nu+1}=c(s_{\nu-1}-s_{\nu})^{-1}\ln\varepsilon_{\nu}^{-1}>  {K}_0 2^{\nu}$$
implies the inequality \begin{eqnarray*}
 2N\leq \sqrt{(\mu_\nu-\mu_{\nu+1}) }N^{3/2}, \quad 4\mu' N^{4}< ( \theta_{\nu+1}- \theta_\nu) {N}^{ 4d\tau_0-1}.
\end{eqnarray*}

\item[$\bullet$ ]
Since the set of Hamiltonians which Poisson commute with $M$ (the momentum) is closed under   Poisson brackets  (or by using Lemma $4.4$ in \cite{GY3}) we $P_{\nu+1}$ satisfies momentum
conservation (namely it Poisson commutes with $M$).
\end{itemize}
% Besides like \cite{GN3}, we have: \beq
%\label{5.6} \|X_{P_{\nu+1}}\|_{D(r_{\nu+1},s_{\nu +1}),\mathcal
%O_{\nu+1}}\le\varepsilon_{\nu+1}. \eeq

  \subsection{Convergence}
Suppose that the assumptions of Theorem \ref{KAM} are satisfied.
 Recall
$$\varepsilon_0=\varepsilon,\,r_0=r,\, s_0=s,\,M_0=M, \,L_0=L, \,\;{\mathcal N}_0={\mathcal N},\, P_0=P,
$$
$\mathcal O$ is an open set. The assumptions of  the iteration
lemma are satisfied when $\nu=0$ if $\varepsilon_0,$ $\gamma$ are
sufficiently small. Inductively, we obtain  sequences:
\[
\mathcal O_{\nu+1}\subset\mathcal O_{\nu},\]
\[\Psi^\nu=\Phi_0\circ\Phi_1\circ\cdots\circ\Phi_\nu:D(r_{\nu+1},s_{\nu+1})\times\mathcal
O_{\nu+1}\to D(r_0,s_0),\nu\ge 0,
\]
\[H\circ\Psi^\nu=H_{\nu+1}={\mathcal N}_{\nu+1}+P_{\nu+1}.\]
\indent Let $\tilde{\mathcal O}=\cap_{\nu=0}^\infty \mathcal O_\nu$,
since at  $\nu$ step the parameter we excluded is bounded by
$C\gamma  {K}_\nu^{-\tau_0+b+d/2}$, the total measure we
excluded with infinity step of KAM iteration is bounded by $\gamma$
which guarantee $\tilde{\mathcal O}$ is a nonempty set, actually it
has positive measure.

As in \cite{P1,P2}, with Lemma {\ref{Lem4.4}},  ${\mathcal N}_\nu,\Psi^\nu,D\Psi^\nu,\omega_{\nu}$ converge uniformly on
$D(0,\frac s 2 )\times\tilde{\mathcal O}$ with
$$
{\mathcal N}_\infty=e_\infty+\la\omega_\infty,I\ra+\sum_{n}\Omega_{n}^\infty
z_n\bar z_n.$$ \noindent Since $
K_\nu=c(s_{\nu-1}-s_{\nu})^{-1}\ln\varepsilon_{\nu}^{-1}$, we have $
\varepsilon_\nu=c\gamma ^2
K^{3\td^2\over\tau_0}_{\nu-1}\varepsilon_{\nu-1}^{\frac {4}{3}}\to
0$ once $\varepsilon$ is sufficiently small. And with this we have
$\omega_\infty$ is slightly different from $\omega$.

Let $\phi_H^t$ be the flow of $X_H$. Since
$H\circ\Psi^\nu=H_{\nu+1}$, there is \beq\label{5.7}
\phi_H^t\circ\Psi^\nu=\Psi^\nu\circ\phi_{H_{\nu+1}}^t. \eeq The
uniform convergence of $\Psi^\nu,D\Psi^\nu,\omega_{\nu}$ and
$X_{H_{\nu}}$ implies that the limits can be taken on both sides of
(\ref{5.7}). Hence, on $D(0,\frac s2)\times\tilde{\mathcal O}$ we
get \beq\label{5.8}
\phi_H^t\circ\Psi^\infty=\Psi^\infty\circ\phi_{H_{\infty}}^t\eeq and
$$
\Psi^\infty:D(0,\frac s2)\times\tilde{\mathcal O}\to
 D(r,s)  \times \mathcal O.
 $$

\noindent From \ref{5.8}, for $\xi\in\tilde{\mathcal O}$,
$\Psi^\infty(\T^b\times\{\xi\})$ is an embedded torus which is
invariant for the original perturbed Hamiltonian system at $\xi\in
\tilde{\mathcal O}$. The normal behavior of this invariant tori is
governed by normal frequency $\Omega_\infty$.
\appendix

\section{Proof of Propositions \ref{smaldiv}, \ref{submain} and \ref{main}}
\subsection{Proposition \ref{smaldiv}}
%  The cases $ {K}\leq K\leq 2 {K}^{\td/\tau_0}$ are proved similarly to
%  Lemma \ref{key}. For the left if a point $m\in [v_i;p_i]_\ell^g$ associate
%with $K_1\geq 2 {K}^{\td/\tau_0}$  we can prove $m\in
%[v_i';p_i']_{\ell'}^g$ associate with $ {K}$, then we get our in the
%same method.

%Let $m$ be any point in  $[v_i;p_i]_\ell^g$ corresponding $K_1$.

%$m$ has a $ {K}$-optical presentation $m\frec{ {K}}[v'_i;p_i']$. With
%Lemma \ref{cutlemma} we have {\it cut} $(\ell',\t1')$ such that
%$|q_{\ell'}|\leq  \theta {{ {K}}^{\t1'}},|q_{\ell'+1}|\geq \mu
%{{ {K}}^{4d\t1'}}$. If $|p_d|\leq { {K}}^{\td\over 4d}$,  by Cramers
%rule $ |m|= |V^{-1}p'|<4{ {K}}^{\td/4d} {K}^{d-1}<4{ {K}}^\td$
%contradict to $|m|\geq 4K_1^\td> 4 {K}^\td$. With Lemma
%\ref{cutlemma} we have $\ell'<d$ and we assume $\ell'$ is the
%smallest.
% Since $B_{ {K}}\subset B_{K_1}$, we have $\ell'\leq \ell$ and
%  $[v'_i;p'_i]_{\ell'}^g \supseteq [v_i;p_i]_\ell^g$.

%  We have just proved if $m\in [v_i;p_i]_\ell^g$ corresponding $K_1$,
%  then there is $[v_i';p_i']_{\ell'}^g$ such that $m\in[v_i';p_i']_{\ell'}^g$ associate
%$ {K}$, a obvious result can be found:
%$$
%|(k,\omega(\xi))+\Omega_{m}-\Omega_{n}|\geq 2\gamma \min(
% {K}^{-2d\tau_0},2^{-4d}|p'_{\ell'}|^{-2d})$$  With this inequality
%and Quasi T\"oplitz of $\sum \Omega_m |z_m|^2$, we get our final
%result like Lemma \ref{key}, also condition $K_1\geq  {K}^{\td\over
%\tau_0}$ is required in the proof.\qed
%\end{proof}

\begin{proof}By hypothesis
$$
|m|,|n|\geq  \theta N^\td\,,\quad m\frec{N}[v_i;p_i]\,,\quad
n\frec{N}[w_i;q_i]\,,$$ \begin{equation}\label{condi}
|q_\ell|,|p_\ell| \leq \mu N^{\t1}\,,\quad
|q_{\ell+1}|,|p_{\ell+1}|\geq  \theta N^{4d\t1}\,,\quad
[v_i;p_i]_\ell\prec [w_i;q_i]_\ell
\end{equation}
By definition of quasi--T\"oplitz  (see Formula \eqref{omet}), one has:
\begin{equation}\label{bohH}
|\tilde\Ome_m-{\it \hat \Omega}([v_i;p_i]_\ell)|,|\tilde\Ome_n-{\it \hat \Omega}([w_i;q_i]_\ell)|\leq
2LN^{-4d\t1}
\end{equation}

Recall that $m-n=-\pi(k)$, so one has $$|m|^2-|n|^2=  \langle m+n,m-n\rangle=|\pi(k)|^2- 2\langle \pi(k),m\rangle.$$

If $\pi(k) \notin \langle v_i\rangle_\ell$ then $|\langle
\pi(k),m\rangle|> N^{4d\t1}> {K}^3$ and the {\em denominator} is
not small:
$$|\langle \ome,k\rangle+ |m|^2-|n|^2 +{\it \hat \Omega}([v_i;p_i]_\ell)-{\it \hat \Omega}([w_i;q_i]_\ell)|> \1 N^{4d\t1},$$ since (again by definition of quasi--T\"oplitz)
$|{\it \hat \Omega}([v_i;p_i]_\ell)|$,$|{\it \hat \Omega}([w_i;q_i]_\ell)|\leq 2L$.

If $\pi(k) \in \langle v_i\rangle_\ell$ then the value of $\langle \pi(k),m\rangle$ is fixed for all $m\in [v_i;p_i]_\ell$.

 We know that $m\frec{ {K}}[v'_i;p'_i]$  has a standard cut, so that  $  m\in [v'_i;p'_i]^g_{\bar \ell}$ for some $\bar\ell$.
If $2^{4d} {K}^{\td}<N^{\tau_0}$ then
%We know that if $|\sigma\Ome_m-\sigma'\Ome_n|> 2|\ome|K_0$ the the denominator is bounded by one. But
%$$ |\sigma\Ome_m-\sigma'\Ome_n|> |\langle \pi(k),m\rangle|-2|\pi(k)|-2L$$ With $\pi(k)<c(b)K_0$, this implies $ |\langle \pi(k),m\rangle|<K_0^2$.  We know that $m$ must  belong to some $ [w_i;q_i]^g$, equipped  with a $K_0$ optimal presentation, and  of rank $\bar\jmath$.  Then there must exist $|q_i|< K_0^2$. If $|q_{\bar\jmath}|>K_0^{\tau_0}$ then (imposing $\tau_0> (4d)^d$) and proceeding by induction, there must exist a $q_{\bar\imath}$ such  that  $|q_{\bar\imath}|<K_0^{\tau_0}$ and $|q_{\bar\imath}|> 2|q_{\bar\imath}|^{4d}$ So that $m\in [\{w_1,\dots,w_{\bar\imath}\};\{q_1,\dots,q_{\bar\imath}\}]^g$. Then, using \eqref{cant1}:
\begin{eqnarray*}& &|\langle \ome,k\rangle+|\pi(k)|^2- 2\langle \pi(k),m\rangle+ { \it\hat\Omega}([v_i;p_i]_\ell)-{ \it\hat\Omega}([w_i;q_i]_\ell)|\\
&\stackrel{\eqref{bohH}}\geq & |\langle \ome,k\rangle+\Omega_{m}-\Omega_{n}|- 4LN^{-4d\t1}\\
&\stackrel{\eqref{bohh2}}\geq &\gamma
\min( {K}^{-2d\tau_0},2^{-4d}|p'_{\bar\ell}|^{-2d})- 4L
|N|^{-4d\t1} \geq  \frac{\gamma}{2}
\min( {K}^{-2d\tau_0},|p'_{\bar\ell}|^{-2d}),
\end{eqnarray*} since $|p'_{\bar\ell}|< 4 {K}^{\td/4d}$ by the definition of standard cut.

If on the other hand we have  $2^{4d} {K}^{\td}>N^{\tau_0}$
we proceed as follows.
% then for  any point $m$ in the affine subspace $[v_i;p_i]_j$ such that \eqref{condi} holds, one has
%%\begin{equation}\label{condi}
%%|m|>  \theta' K^\td\,,  |(m,\ell)|>  \theta' (\mu')^{-1}\max(\mu' K^{4d\tau_0},|p_j|^{4d})\,,\; \forall \ell\in B_K^a\setminus \langle v_i\rangle_j,
%%\end{equation}  then
%\begin{equation}\label{bohh3} |\langle \ome,k\rangle+\Ome_m-\Ome_n|\geq \gamma \min(K^{-4d \tau_0}, 2^{-4d}|p_j|^{-4d}).
%\end{equation}
We have seen that we may restrict to the case $\pi(k)\in \langle v_i\rangle_j$,
 where  $$|m|^2-|n|^2= |\pi(k)|^2-2\langle \pi(k),m\rangle=|\pi(k)|^2-2(\pi(k),m^g),$$ where
 (notice that  $N<  2 {K}^{\td/\tau_0}$), $m^g:= m^g(N)$ is the point in $[v_i;p_i]_\ell^g$ chosen for the measure estimates \eqref{third}.

 We notice that $m^g,n^g$ satisfy the conditions \eqref{condi}, so we apply \eqref{bohH} to $m,n,m^g,n^g$. We have
$$ |\langle \ome,k\rangle+|\pi(k)|^2 - 2\langle \pi(k),m\rangle+{\it \hat \Omega}([v_i;p_i]_\ell)-{\it\hat\Ome}([w_i;q_i]_\ell)|  $$
$$\geq|\langle \ome,k\rangle+{ \Ome}_{m^g}-{\Ome}_{n^g}|-4L N^{-4d\t1}$$ $$\geq {\gamma\over2} \min(N^{-2d \tau_0}, 2^{-2d}|p_\ell|^{-2d}) -4 L N^{-4d\t1}$$
$$\geq{\gamma \over 4}\min(N^{-2d \tau_0}, 2^{-2d}|p_\ell|^{-2d})\gtrdot \gamma  {K}^{\frac{-2d\t1\td}{\tau_0}}
$$
since by definition $|p_j|< \mu' N^{\t1}<4 N^{\t1}$,  $ N\leq
2 {K}^{\td/\tau_0}$. \end{proof}

 \subsection{Proposition \ref{submain}}
{\em Proof.} \ The quasi--T\"oplitz  property is a condition on the
$(N, \tet1,\mu,\tau)$--bilinear part of $F$, where $F$ is at most
quadratic. Hence we   only need to consider the quadratic terms:
\begin{equation} \Pi_{(N, \tet1,\mu,\tau)}F=\!\!\!\!\!\!\!\!\!\!\!\!\!\!\! \sum_{{ |k|<N\,,\;|m|,|n| > \theta N^{\tau_1}\atop\exists \ell:\; m,n \,{\rm have \;a }\, \ell\, {\rm cut}}\atop{{\rm with\, parameters}\, N, \tet1,\mu,\t1 }}\!\!\!\!\!\!\!\!\!\!e^{\ii
\langle k, \teta\rangle}(F_{k,0,e_m,e_n} z_m\bar z_n + F_{k,0,e_m+e_n,0} z_mz_n) \,+\,
{\rm c.c}\,.\end{equation} Recall that
\begin{equation}\label{homol}
F_{k,0,e_m,e_n}=  \frac{P_{k,0,e_m,e_n}}{\langle k,\omega\rangle+\Omega_m-\Omega_n}\,,\quad F_{k,0,e_m+e_n,0}=  \frac{P_{k,0,e_m+e_n,0}}{\langle \ome,k\rangle+\Omega_m+\Omega_n}.
\end{equation}
By hypothesis $|m|,|n|>  \theta N^\td$ so in the case of
$F_{k,0,e_m+e_n,0}$ one has
$$ |F_{k,0,e_m+e_n,0}|= \frac{|P_{k,0,e_m+e_n,0}|}{\langle k,\omega\rangle+|m|^2+|n|^2 +\tilde\Omega_m+\tilde\Omega_n}\leq |P_{k,0,e_m+e_n,0}| N^{-\td}, $$ since
$$| \langle k,\omega\rangle+|m|^2+|n|^2 +\tilde\Omega_m+\tilde\Omega_n|> 2N^{\td}- c {K}-2L.$$ We proceed in the same way for $\partial_\xi F_{k,0,e_m+e_n,0}$.
This means that $F_{k,0,e_m+e_n,0}$ is quasi-T\"oplitz with the ``T\"oplitz approximation'' equal to zero.
Recalling that $P$ is quasi-T\"oplitz we deduce, by Remark \ref{pollen}, that if $m\frec{N}[v_i;p_i]$, $|m|,|n|>\theta N^{\td}$ and $m,n$ have a cut $\ell,\tau$, then  we have:
$$
P_{k,0,e_m,e_n}= \mathcal P_k(m-n,[v_i;p_i]_\ell)+ N^{-4d\t1} \bar P_{k,0,e_m,e_n}\,.
$$
Note that by definition (see formula \eqref{scotto}) for all $m,n$ which have a $\ell,\tau$ cut the T\"oplitz approximation ${\mathcal P}_{k}(m-n,[v_i;p_i]_\ell)$ must depend only on $m-n$ on the affine subspace $[v_i;p_i]_\ell$ and on $k$. Moreover the approximation \eqref{billa} must hold for all $m\in  [v_i;p_i]_\ell$ which have a cut $\ell,\tau$ (naturally if we fix $\tau$ and an affine subspace $[v_i;p_i]_\ell$ it may well be possible that no integer point $m\in  [v_i;p_i]_\ell$ has a cut $\ell,\tau$).

Finally  since $\sum_m \tilde\Ome_m z_m \bar z_m$ is quasi-T\"oplitz, diagonal and quadratic we have:
$$
\tilde \Ome_m = { \it\hat\Omega}([v_i;p_i]_\ell)+ N^{-4d\t1} \bar \Ome_m
$$
for all $m\frec{N}[v_i;p_i]$ which have an $\ell,\t1$ cut.

 We wish to show
that
\begin{equation}\label{billa}
 F_{k,0,e_m,e_n}={\mathcal F}_{k}(m-n,[v_i;p_i]_\ell)+
N^{-4d\t1}\bar F_{k,0,e_m,e_n},
\end{equation}
  here ${\mathcal F}_{k}$ is the
$k$ Fourier coefficient of the T\"oplitz approximation $\mathcal F$.

By hypothesis we have conditions \eqref{condi} and $\langle
v_1,\cdots,v_\ell\rangle= \langle w_1,\cdots,w_\ell\rangle$. This in
turn implies  that the subspace $[w_i,q_i]_\ell$
  is obtained from  $[v_i,p_i]_\ell$ by translation by  $m-n=-\pi(k)$.
  If $\pi(k)\notin \langle v_i\rangle_\ell$ then the denominator in the first of \eqref{homol} is
  $$ |\langle k,\omega\rangle+\Omega_m-\Omega_n|> |\langle k,\omega\rangle+|\pi(k)|^2- 2\langle \pi(k),m\rangle |- 2L> \frac14 N^{4d\t1}$$ and we may again set $\mathcal F_k(m-n,[v_i,p_i]_\ell) =0$. Otherwise we set

 $$ {\mathcal F}_{k}(m-n,[v_i,p_i]_\ell)=  \frac{{\mathcal P}_{k}(m-n,[v_i,p_i]_\ell)}{\langle \ome,k\rangle+|\pi(k)|^2- 2\langle \pi(k),m\rangle+{ \it\hat\Omega}([v_i;p_i]_\ell)-{ \it\hat\Omega}([w_i;q_i]_\ell)}.$$
We notice that $\langle \pi(k),m\rangle$ depends only on the subspace $[v_i,p_i]_\ell$ and on $\pi(k)$.
Moreover by definition ${ \it\hat\Omega}(\cdot)$ depends only on the affine subspace on which it is computed; finally $[w_i;q_i]_\ell$ depends only on $[v_i;p_i]_\ell$ and on $k$.  Hence $ {\mathcal F}_{k}(m-n,[v_i,p_i]_\ell)$ depends only on $k, m-n$ and $[v_i,p_i]_\ell$ as was our claim.  Finally we apply Proposition \ref{smaldiv} to bound the denominator.
In order to bound the derivatives in $\xi$ of $F$ we proceed in the same way,
only the denominators may appear to the power two.

Finally to bound $\bar F$ we notice that
$${\bar F}_{k,m,n}= \frac{{\bar P}_{k,m,n}}{\mathcal D}+N^{4d\t1}{\mathcal P}_{k}(m-n,[v_i,p_i]_\ell) \frac{\tilde\Ome_m -{ \it\hat\Omega}([v_i;p_i]_\ell)-\tilde\Ome_n+ { \it\hat\Omega}([w_i;q_i]_\ell)}{D \mathcal D} $$ where
$$ \mathcal D= \langle \ome,k\rangle+|\pi(k)|^2- 2\langle \pi(k),m\rangle+{ \it\hat\Omega}([v_i;p_i]_\ell)-{ \it\hat\Omega}([w_i;q_i]_\ell)\,,\quad D= \langle \ome,k\rangle +\Ome_m-\Ome_n,$$ and $   {N}^{4d\t1}|\tilde\Ome_m -{ \it\hat\Omega}([v_i;p_i]_\ell)|\leq  2L$. In conclusion taking the $\sup_{N>K,\t1<\td}$:
$$\|X_F\|^T_{r,s}\leq C\gamma^{-2}
N^{\frac{3\td^2}{\tau_0}}\|X_P\|^T_{r,s}\eqno\endproof$$ 
\subsection{Proposition \ref{main}}
Before  proving Proposition \ref{main}, we discuss some technical Lemma and
 set up some notation. We divide the Poisson bracket in four terms:
 $\{\cdot,\cdot\}= \{\cdot,\cdot\}^{I,\teta}+\{\cdot,\cdot\}^{L}+\{\cdot,\cdot\}^{H}+\{\cdot,\cdot\}^{R}$ where the superscript $L,H,R$  identifies the variables in which we are performing the derivatives (the symbol $R$ summarizes the derivatives in all the $w_i$ which are neither low nor high momentum).
 We call a monomial $$e^{i\langle k,\teta\rangle}I^l z^\alpha\bar z^\beta$$

 1. of $(N,\mu)$-low momentum if $|k|<N$ and $\sum_j |j| (\alpha_j+\beta_j)< \mu N^3$.  Denote by $\Pi^L_{N,\mu}$ the projection on this subspace.

 2. of $N$-high frequency if $|k|\geq N$. Denote $\Pi^U_N$ the projection on this subspace.

  Recall that the projection symbol $\Pi_{N, \tet1,\mu,\t1}$  is given in definition \ref{taubilinear}.  A function $f$ then may be uniquely represented as
  $f= \Pi_{N, \tet1,\mu,\t1} f+ \Pi_{N,\mu}^L f + \Pi_N^U f + \Pi_R f$ where $ \Pi_R f$ is by definition the projection on those monomials which
  are neither $(N, \tet1,\mu,\t1)$ bilinear nor of $(N,\mu)$-low momentum nor of $N$-high frequency.

  A technical lemma is given below.
\begin{Lemma}
  The following splitting formula holds:
%  \begin{equation}\label{split}
%\Pi_{K, \theta',\mu'}\{f,g \}= \Pi_{K, \theta',\mu'}\left(  \{\Pi^H_{K, \tet1,\mu} f^{(1)},\Pi^H_{K, \tet1,\mu} f^{(2)}\}^H+\right.\end{equation} $$ \{\Pi^H _{K, \tet1,\mu}f^{(1)},\Pi^L_{K,\mu} f^{(2)}\}^{I,\teta}+\{\Pi^H _{K, \tet1,\mu}f^{(1)},\Pi^L_{K,\mu} f^{(2)}\}^L+ \{\Pi^U_K f^{(1)},f^{(2)}\}$$ $$\left.
%\{\Pi^L_{K,\mu} f^{(1)},\Pi^H_{K, \tet1,\mu} f^{(2)}\}^{I,\teta}+\{\Pi^L_{K,\mu} f^{(1)},\Pi^H_{K, \tet1,\mu} f^{(2)}\}^L+\{f^{(1)},\Pi^U_K f^{(2)}\}\right) $$
%\textbf{
%(make a change, take away "H" see symbol in definition 2.2,  left to be done in this way)}
\begin{equation}\label{split}
\Pi_{N, \theta',\mu',\t1}\{f^{(1)},f^{(2)}   \}=
\Pi_{N, \theta',\mu',\t1}\left(  \{\Pi_{N, \tet1,\mu,\t1}
f^{(1)},\Pi_{N, \tet1,\mu,\t1} f^{(2)}\}^H+\right.\end{equation}
$$ \{\Pi _{N, \tet1,\mu,\t1}f^{(1)},\Pi^L_{N,2\mu}
f^{(2)}\}^{I,\teta}+\{\Pi
_{N, \tet1,\mu,\t1}f^{(1)},\Pi^L_{N,2\mu} f^{(2)}\}^L+ \{\Pi^U_N
f^{(1)},f^{(2)}\}$$ $$\left. \{\Pi^L_{N,2\mu}
f^{(1)},\Pi_{N, \tet1,\mu,\t1}
f^{(2)}\}^{I,\teta}+\{\Pi^L_{N,2\mu}
f^{(1)},\Pi_{N, \tet1,\mu,\t1} f^{(2)}\}^L+\{f^{(1)},\Pi^U_N
f^{(2)}\}\right) $$
\end{Lemma}
\begin{proof}
We %divide $f^{(i)}= \Pi_{K, \tet1,\mu} f^{(i)}+ \Pi_{K,\mu}^L f^{(i)} + \Pi_K^U f^{(i)} + \Pi_R f^{(i)}$ and divide the Poisson brackets as  $\{\cdot,\cdot\}= \{\cdot,\cdot\}^{I,\teta}+\{\cdot,\cdot\}^{L}+\{\cdot,\cdot\}^{H}+\{\cdot,\cdot\}^{R}$, then we
perform a case analysis: we replace each $f^{(i)}$ with  a single
monomial to show which terms may contribute  non trivially to the
projection $\Pi_{N, \theta',\mu',\t1}\{f^{(1)},f^{(2)}   \}$.

 Consider the expression
  $$ \Pi_{N, \theta',\mu',\t1}\{e^{i \langle k^{(1)},\teta\rangle}I^{l^{(1)}} z^{\alpha^{(1)}}\bar z^{\beta^{(1)}},e^{i \langle k^{(2)},\teta\rangle }I^{l^{(2)}} z^{\alpha^{(2)}}\bar z^{\beta^{(2)}}\}.$$
 If one or both of the $|k^{(i)}|>N$ then one or both monomials are of high frequency and we obtain the last term in the second and third line of \eqref{split}.

 Suppose now that  $|k^{(1)}|,|k^{(2)}|<N$ we wish to understand under which conditions on the $\alpha^{(i)},\beta^{(i)}$ this expression is not zero. By direct inspection, one of the following situations (apart from a trivial permutation of the indexes $1,2$) must hold:
  \begin{enumerate}
  \item  one has $z^{\alpha^{(1)}}\bar z^{\beta^{(1)}}= z^{\bar\alpha^{(1)}}\bar z^{\bar\beta^{(1)}} z_m^{\sigma}z_j^{\sigma_1}$ and $ z^{\alpha^{(2)}}\bar z^{\beta^{(2)}}= z^{\bar\alpha^{(2)}}\bar z^{\bar\beta^{(2)}}z_n^{\sigma'} z_j^{-\sigma_1}$,
   where    $|m|,|n|\geq  \theta' N^\td$ have a  cut for some $\ell$ with parameters $(N,\theta',\mu',\tau)$  and
  $z^{\bar\alpha^{(1)}}\bar z^{\bar\beta^{(1)}}z^{\bar\alpha^{(2)}}\bar z^{\bar\beta^{(2)}}$ is of $(N,\mu')$--low momentum.
  The derivative in the Poisson bracket is on $w_j$;

\item  one has $z^{\alpha^{(1)}}\bar z^{\beta^{(1)}}= z^{\bar\alpha^{(1)}}\bar z^{\bar\beta^{(1)}} z_m^{\sigma}z_n^{\sigma'}$ and
  $ z^{\alpha^{(2)}}\bar z^{\beta^{(2)}}= z^{\bar\alpha^{(2)}}\bar z^{\bar\beta^{(2)}} $, where    $|m|,|n|\geq  \theta' N^\td$ have a cut for some $\ell$ with parameters $(N,\theta',\mu',\tau)$and
  $z^{\bar\alpha^{(1)}}\bar z^{\bar\beta^{(1)}}z^{\bar\alpha^{(2)}}\bar z^{\bar\beta^{(2)}}$ is of $(N,\mu')$--low momentum. The derivative in the Poisson bracket is on $I,\teta$;

  \item one has $z^{\alpha^{(1)}}\bar z^{\beta^{(1)}}= z^{\bar\alpha^{(1)}}\bar z^{\bar\beta^{(1)}} z_m^{\sigma}z_n^{\sigma'}z_j^{\sigma_1}$ and
  $ z^{\alpha^{(2)}}\bar z^{\beta^{(2)}}= z^{\bar\alpha^{(2)}}\bar z^{\bar\beta^{(2)}} z_j^{-\sigma_1}$ where
    $|m|,|n|\geq  \theta' N^\td$ have a cut for some $\ell$ with parameters $(N,\theta',\mu',\tau)$ and
  $z^{\bar\alpha^{(1)}}\bar z^{\bar\beta^{(1)}}z^{\bar\alpha^{(2)}}\bar z^{\bar\beta^{(2)}}$ is of $(N,\mu')$--low momentum. The derivative in the Poisson bracket is on $w_j$;

  \item one has $z^{\alpha^{(1)}}\bar z^{\beta^{(1)}}= z^{\bar\alpha^{(1)}}\bar z^{\bar\beta^{(1)}} z_m^{\sigma}$ and
  $ z^{\alpha^{(2)}}\bar z^{\beta^{(2)}}= z^{\bar\alpha^{(2)}}\bar z^{\bar\beta^{(2)}} z_n^{\sigma'}$ where    $|m|,|n|\geq  \theta' N^\td$ have a cut for some $\ell$ with parameters $(N,\theta',\mu',\tau)$ and
  $z^{\bar\alpha^{(1)}}\bar z^{\bar\beta^{(1)}}z^{\bar\alpha^{(2)}}\bar z^{\bar\beta^{(2)}}$ is of $(N,\mu')$--low momentum. The derivative in the Poisson bracket is on $I,\teta$.
  \end{enumerate}

Case 1. We apply momentum conservation to both monomials and
obtain $$\sigma_1j=-\sigma
m-\pi(k^{(1)},\bar\alpha^{(1)},\bar\beta^{(1)})= \sigma'
n+\pi(k^{(2)},\bar\alpha^{(2)},\bar\beta^{(2)}).$$  Recall that
$$ \sum_{l\in \Z_1^d}|l|(\bar\alpha^{(1)}_l+\bar\beta^{(1)}_l+\bar\alpha^{(2)}_l+\bar\beta^{(2)_l})\leq \mu' N^3\;\longrightarrow \sum_{l\in \Z_1^d}|l|(\bar\alpha^{(i)}_l+\bar\beta^{(i)}_l)\leq \mu' N^{\td} $$
and by hypothesis $|k^{(i)}|\leq N$, this implies that $|j|>
 \theta' N^\td- \mu' N^3-CN>  \theta N^\td$ for $N> {K}'$ respecting
\eqref{pois1} (recall that $C$ is a constant so that $|\pi(k)|\leq C|k|$).
Hence $\min(|m|,|n|,|j|)> \theta N^\td$.
  By momentum conservation $|\sigma m+\sigma_1 j|, |-\sigma_1 j+\sigma' n|\leq CN +\mu' N^3\leq 5 N^3$; by hypothesis $n,m$ have a cut  $\ell$ with parameters $(N,\theta',\mu',\tau)$.
  By Lemma \ref{mah} also $j\frec{N}[w_i;q_i]$ has a  cut $\ell$ with  parameters $(N,\theta,\mu,\tau)$.
 Then $e^{i (k^{(i)},\teta)} z^{\alpha^{(i)}}\bar z^{\beta^{(i)}}$ are by definition $(N, \tet1,\mu,\t1)$ bilinear. The derivative in the Poisson bracket is on $j$ which is a high momentum variable.

  As $m,n$ run over all possible vectors in $\Z^d_1$ with $|m|,|n|\geq  \theta' N$, we obtain the first term in formula \eqref{split}.

Case 2.  Following the same argument  $e^{i \langle k^{(1)},\teta\rangle }
z^{\alpha^{(1)}}\bar z^{\beta^{(1)}}$ is  $(N, \theta',\mu',\t1)$
bilinear and $e^{i \langle k^{(2)},\teta\rangle} z^{\alpha^{(2)}}\bar
z^{\beta^{(2)}}$ is $(N,\mu')$ low momentum. We obtain the second
contribution in formula \eqref{split}.

  Case 3.  We apply momentum conservation to the second monomial and obtain $-\sigma_1j=-\pi(k^{(2)},\bar\alpha^{(2)},\bar\beta^{(2)})$.
  This implies that $$|j|+ \sum_{l\in \Z^d_1} |l|(\bar\alpha^{(1)}_l+\bar\beta^{(1)}_l)\leq |\pi(k^{(2)},\bar\alpha^{(2)},\bar\beta^{(2)}|+ \sum_{l\in \Z^d_1} |l|(\bar\alpha^{(1)}_l+\bar\beta^{(1)}_l)\leq $$
  $$ CN + \sum_{ l\in \Z^d_1} |l|(\bar\alpha^{(1)}_l+\bar\beta^{(1)}_l+\bar\alpha^{(2)}_l+\bar\beta^{(2)}_l )\leq \mu' N^3+ CN\leq \mu N^3 $$ if $N> {K}'$ with $ {K}'$ satisfying \eqref{pois1}.
  Then $e^{i \langle k^{(1)},\teta\rangle} z^{\alpha^{(1)}}\bar z^{\beta^{(1)}}$ is,
  by definition, $(N, \tet1,\mu,\t1)$ bilinear and $e^{i \langle k^{(2)},\teta\rangle} z^{\alpha^{(2)}}\bar z^{\beta^{(2)}}$ is
  $(N,2\mu)$ low momentum. The derivative in the Poisson bracket is on $j$ which is a low momentum variable. We obtain the third  contribution in formula \eqref{split}.

Case 4. We apply momentum conservation to both monomials, we get
$$\min(|\sigma m|,|\sigma' n|)\leq \max_{i=1,2}( |-\pi(k^{(i)},\bar\alpha^{(i)}, \bar\beta^{(i)})|\leq CN+\mu' N^3,$$ which is in contradiction to the hypothesis $|m|,|n|\geq  \theta' N^\td$. Hence case 4. does not give any contribution.

 The third line in formula \eqref{split} is dealt just as the second line by exchanging the indexes $1,2$.
\end{proof}

% \begin{Remark}
% By momentum conservation the case $z^{\alpha^{(1)}}\bar z^{\beta^{(1)}}= z^{\bar\alpha^{(1)}}\bar z^{\bar\beta^{(1)}} z_m^{\sigma}$ and
%  $ z^{\alpha^{(2)}}\bar z^{\beta^{(2)}}= z^{\bar\alpha^{(2)}}\bar z^{\bar\beta^{(2)}}z_n^{\sigma'} $  (the derivative in the Poisson brackets being on $I,\teta$) is not allowed.
%\end{Remark}

In order to show that $\{f^{(1)},f^{(2)}\}$ is quasi--T\"oplitz, for
all $N> {K}'$ and $\t1$ we have to provide a decomposition
$$\Pi_{N, \theta',\mu',\t1} \{f^{(1)},f^{(2)}\}= \mathcal F^{(1,2)}+ N^{-4d \t1} \bar f^{(1,2)}$$  so that  $\mathcal F^{(1,2)}\in \mathbb F$ and
\begin{equation}\label{stizza}
\|X_{\mathcal F^{(1,2)}}\|_{r',s'}, \|X_{\bar f^{(1,2)}}\|_{r',s'} <
\del^{-1}C\|X_{f^{(1)}}\|_{r,s}^T\|X_{f^{(1)}}\|_{r,s}^T.
\end{equation}
for some constant $C$.

Using Remark \ref{osserv}, we substitute in  formula \eqref{split} $
\Pi_{N, \theta',\mu',\t1}f^{(i)}= \mathcal F^{(i)}+
N^{-4d\t1}\bar f^{(i)}$, with $\mathcal F^{(i)}\in \mathbb
F$.
\begin{Lemma}
Consider the function $$\mathcal F^{(1,2)}=
\Pi_{N, \theta',\mu',\t1}\left(  \{ \mathcal F^{(1)},\mathcal
F^{(2)}\}^H+ \{\mathcal F^{(1)},\Pi^L_{N,2\mu}
f^{(2)}\}^{(I,\teta)+L}+\{\Pi^L_{N,2\mu} f^{(1)},\mathcal
F^{(2)}\}^{(I,\teta)+L} \right)$$ where we have denoted
$\{\cdot,\cdot\}^{(I,\teta) +L}=
\{\cdot,\cdot\}^{(I,\teta)}+\{\cdot,\cdot\}^{L}$.   (i) One has
$\mathcal F^{(1,2)}\in \mathbb F$. (ii) Setting $\bar
f^{(1,2)}= N^{4d\t1}(\Pi_{N, \theta',\mu',\t1}
\{f^{(1)},f^{(2)}\}-\mathcal F^{(1,2)})$ one has that the bounds
\eqref{stizza} hold.
\end{Lemma}
 \begin{proof}
  In order to prove the first statement it is useful to write
   $$\mathcal F^{(i)}= \sum_{A=[v_i;p_i]_\ell\in \mathcal H_N\atop |p_\ell|<\mu N^{\t1}}\sum_{\sigma,\sigma'=\pm 1}\sum_{m,n }^{(N,\theta',\mu',\t1,A)}[\mathcal F^{(i)}]^{\sigma,\sigma'}(I,\teta,w^L;\sigma m+\sigma' n,[v_i;p_i]_\ell) z_m^{\sigma}z_n^{\sigma'}$$
  where $\sum\limits^{(N,\theta',\mu',\t1,A)}$ is the sum over those $n,m$ which respect \eqref{zerouno} and have the $\ell$ cut  at $A= [v_i;p_i]_\ell$ with the parameters $ \theta',\mu',\tau$.
  For compactness of notation we will omit the dependence on $(I,\teta,w^L)$.

 The fact that $\{\mathcal F^{(1)},\Pi^L_{N,2\mu} f^{(2)}\}^{I,\teta+L} \in \mathbb F$ is obvious. Indeed the coefficient of $z_m^\sigma z_n^{\sigma'}$ is
 $$ \{\mathcal F^{(1)}(\sigma m+\sigma' n,[v_i;p_i]_\ell),\Pi^L_{N,2\mu} f^{(2)}\}^{I,\teta+L},$$ the same for $\{\mathcal F^{(2)},\Pi^L_{N,2\mu} f^{(1)}\}^{I,\teta+L}.$

Suppose now that $n,m$  respect \eqref{zerouno} and have the
$\ell$ cut  $[v_i;p_i]_\ell$ with the parameters
$ \theta',\mu',\tau$. By the rules of Poisson brackets the coefficient of
$z_m^\sigma z_n^{\sigma'}$ in the expression $ \{ \mathcal
F^{(1)},\mathcal F^{(2)}\}^H$ is
   \begin{equation}\label{ahi}\sum_{r\in \Z^d_1,\sigma_1=\pm 1\atop{ |r|\geq  \theta N^\td \atop {|\sigma m+\sigma_1r|\leq \mu N^3\atop |-\sigma_1r+\sigma'n|\leq \mu N^3}}}\!\!\!\!{-\sigma_1} [{  \mathcal F}^{(1)}]^{\sigma,\sigma_1}(\sigma m+\sigma_1r,[v_i;p_i]_\ell)[{ \mathcal F}^{(2)}]^{-\sigma_1,\sigma'}(-\sigma_1r+\sigma'n;[w_i;q_i]_\ell); \end{equation}
Since $|\sigma m+ \sigma_1 r|, |\sigma' n-\sigma_1 r|\leq \mu N^3$
and $|m|,|n|> \theta' N^\td$ we have that the condition
$|r|> \theta N^\td$ is automatically fulfilled. By Lemma \ref{mah}
$r,n,m$ all have a $\ell$ cut with parameters
$( \tet1,\mu,\tau)$. We set $m\frec{N}[v_i;p_i]$,
$n\frec{N}[v'_i;p'_i]$, $r\frec{N}[w_i;q_i]$.  Again by Lemma \ref{mah} $\langle v_i\rangle_\ell=\langle v'_i\rangle_\ell=\langle w_i\rangle_\ell$, moreover $[w_i;q_i]_\ell$ is completely fixed by $[v_i;p_i]_\ell$, $\sigma,\sigma_1$ and by $\sigma m+ \sigma_1 r:= h$.  We may  suppose (the other cases are done in the same way) that
  $$(p_1,\cdots,p_\ell,v_1,\cdots,v_\ell)\preceq(q_1,\cdots,q_\ell,w_1,\cdots,w_\ell)\preceq (p'_1,\cdots,p'_\ell,v'_1,\cdots,v'_\ell),$$
  note that also this order relation depends only on $\s,\s',\s_1$, $[v_i;p_i]_\ell$, $\s m+\s' n$ and $\sigma m+ \sigma_1 r=h$.
  Then we may change variables in the sum over $r$ in \eqref{ahi}:
 $$\sum_{\sigma_1=\pm 1}\sum_{h\,: |h| <\mu N^3\atop |\sigma m+\sigma' n -h|\leq \mu N^3 }-{\sigma_1} [{  \mathcal F}^{(1)}]^{\sigma,\sigma_1}(h,[v_i;p_i]_\ell)[{ \mathcal F}^{(2)}]^{-\sigma_1,\sigma'}(\sigma m+\sigma' n -h;[w_i;q_i]_\ell),$$ this expression only depends on $[v_i;p_i]_\ell$. The estimate \eqref{stizza} for $\mathcal F^{(1,2)}$ follows by  Cauchy  estimates since
 $$ \|X_{\mathcal F^{(1,2)}}\|_{r',s'}\leq  \|X_{\{\mathcal F^{(1)},\mathcal F^{(2)}\}}\|_{r',s'}+ \|X_{\{\mathcal F^{(1)},f^{(2)}\}}\|_{r',s'}+ \|X_{\{\mathcal F^{(2)},f^{(1)}\}}\|_{r',s'}.$$

We now compute:
$$\bar f= \Pi_{N, \theta',\mu',\t1}\left( \{ \Pi _{N, \tet1,\mu,\t1} f^{(1)},\bar f^{(2)}\}^H+\{\bar f^{(1)} , \mathcal F^{(2)}\}^H\right.$$ $$ +
\{\bar f^{(1)},\Pi^L_{N,\mu} f^{(2)}\}^{I,\teta}+\{\bar
f^{(1)},\Pi^L_{N,\mu} f^{(2)}\}^L+ N^{4d\t1}\{\Pi^U_N
f^{(1)},f^{(2)}\}$$ $$\left. \{\Pi^L_{N,\mu} f^{(1)},\bar
f^{(2)}\}^{I,\teta}+\{\Pi^L_{N,\mu} f^{(1)},\bar
f^{(2)}\}^L+N^{4d\t1}\{f^{(1)},\Pi^U_N f^{(2)}\}\right).  $$

 Since $e^{-N(s-s')}<N^{-\td}$, one has
 $$\|  X_{ \{f^{(1)},\Pi^U_N f^{(2)}\}}\|_{r',s'} \leq  N^{-\td} 2^{2d+1}\delta^{-1}\|X_{ f^{(1)}}\|_{r,s}\|X_{f^{(2)}}\|_{r,s}, $$ by the Cauchy and smoothing estimates. The estimate \eqref{stizza} follows.
 \end{proof}
 \begin{proof}{\it  (\textbf{Proposition \ref{main}})}
 Proposition \ref{main}(i) follows from the previous Lemma.

(ii)   Given $f^{(i)}$, $ i=1,\cdots,J$ as in item (i), and applying
repeatedly \eqref{pois1},  the nested Poisson bracket
$$\{f^{(1)},\{f^{(2)},\cdots,\{f^{(J-1)},f^{(J)}\}\cdots\}$$ is quasi-T\"oplitz in $\D(r_+,s_+)$  with parameters $( {K}_+, \theta_+,\mu_+)$ if
\begin{equation}\label{enne}
 \frac{1}{N^2}\leq \frac{(\mu-\mu') }{J}, \quad \frac{2\mu' }{N^{4d \tau_0-4}}< \frac{ \theta'- \theta}{J}\,,\quad  e^{-\frac{s-s'}{J}N}(N)^\td<1
\end{equation}
 for all $N> {K}_+$
\vskip5pt For given $N$ we bound all the terms in $e^{\{F,\cdot\}}G$
containing $J>(\ln N)^2$ Poisson brackets by $N^{-\td}$ by using
the standard bound:
$$\sum_{k>J}\frac{\|X_{ ad(f^{(1)})^k f_2}\|_{r',s'} }{k!}\leq (2e\del^{-1} \|X_{f^{(1)}}\|_{r,s})^{J+1}\|X_{f^{(2)}}\|_{r,s}\leq $$ $$CN^{-\td} \|X_{f^{(1)}}\|_{r,s}\|X_{f^{(2)}}\|_{r,s}$$ provided that $ 2e\del^{-1} \|X_{f^{(1)}}\|_{r,s}<\frac12$.
 We then apply \eqref{enne} with $J= (\ln N)^2$, we get the
restriction \eqref{boh}. So applying item (i) repeatedly we get for all $k<J$:
$$ \frac{1}{k!}\|X_{ ad(f^{(1)})^k f_2}\|^T_{r',s'}\leq (Ce\del^{-1} \|X_{f^{(1)}}\|^T_{r,s})^{k}\|X_{f^{(2)}}\|_{r,s},$$ the result follows.
 \end{proof}

\medskip
{\it Acknowledgments:} \ We wish to thank Massimiliano Berti,  Luca Biasco and Claudio Procesi for their careful  reading and many helpful suggestions.  Finally we wish to thank the anonymous referees whose remarks helped us to make the paper more readable.

% pointed out many misprints  and inconsistencies in the previous version.

%%%%%%%%%%%%%%%%%%%%%%%%%%%%%%%%%%%%%%%%%%%%%%%%%%%%%%%%%%%%%%%%%%

\end{document}